\documentclass[10pt]{amsart}

\usepackage{txfonts}
\usepackage[T1]{fontenc}
\usepackage{fancyhdr}

\usepackage{geometry} 
\usepackage{booktabs} 
\usepackage{array} 
\usepackage{paralist} 
\usepackage{verbatim} 
\usepackage{tikz}
\usetikzlibrary{arrows}
\usetikzlibrary{decorations.markings}
\usepackage{subfig}
\usepackage{tabularx}
\usepackage{enumitem}
\usepackage{amsmath,amsthm,amscd}
\usepackage{amssymb,amsfonts}
\usepackage{fancyhdr} 
\usepackage{pgf,tikz}
\usepackage{caption}
\usepackage{tikz-cd}
\usepackage{tikzscale}
\usetikzlibrary{patterns}
\usepackage{rotating}
\usepackage{xcolor}
\usepackage{lscape}
\usepackage{xspace}
\usepackage{xfrac}
\usepackage{blindtext}
\usepackage[ruled,vlined]{algorithm2e}

\usetikzlibrary{decorations.markings}
\usetikzlibrary{patterns}
\usetikzlibrary{calc}
\usepackage{soul}

\usepackage{tikz}
        \usetikzlibrary{patterns,hobby}
        \usepackage{pgfplots}
        \pgfplotsset{compat=1.6}

\usepackage{faktor} 
\usepackage{pgf,tikz}
\usetikzlibrary{arrows.meta}
\usetikzlibrary{decorations.markings}
\usetikzlibrary{patterns}

\usepackage{color}   
\usepackage{hyperref}
\hypersetup{
    colorlinks=true, 
    linktoc=all,     
    linkcolor=blue,  
    citecolor=blue,
    filecolor=blue,
    urlcolor=blue
}

\usepackage[T1]{fontenc}
\usepackage{hyperref}
\usepackage{nameref}

\usepackage{amssymb,amsmath,latexsym,amsthm}
\usepackage{faktor} 		

\theoremstyle{plain}                    
\newtheorem{thm}{Theorem}[section]
\newtheorem*{kthm}{Kronecker's Approximation Theorem}
\newtheorem*{rthm}{Ratner's Theorem}
\newtheorem{thma}{Theorem}
\newtheorem{propa}[thma]{Proposition}

\newtheorem{lem}[thm]{Lemma}
\newtheorem{prop}[thm]{Proposition}
\newtheorem{cor}[thm]{Corollary}
\newtheorem*{thmnn}{Theorem}
\newtheorem*{propnn}{Proposition \ref{aiffd}}

\theoremstyle{definition}
\newtheorem{defn}[thm]{Definition}
\newtheorem{ex}[thm]{Example}

\theoremstyle{remark}
\newtheorem{rmk}[thm]{Remark}

\numberwithin{equation}{section}

\newcommand{\R}{\mathbb{R}}

\newcommand{\C}{\mathbb{C}}
\newcommand{\Z}{\mathbb{Z}}

\newcommand{\Q}{\mathbb{Q}}

\newcommand{\T}{\mathbb{T}}

\newcommand{\slc}{\mathrm{SL}(2,\C)}
\newcommand{\pslr}{\mathrm{PSL}(2,\R)}

\newcommand{\slz}{\mathrm{SL}(2,\Z)}
\newcommand{\slr}{\mathrm{SL}(2,\R)}
\newcommand{\spzz}{\mathrm{Sp}(2,\Z)}
\newcommand{\spz}{\mathrm{Sp}(2g,\Z)}
\newcommand{\spr}{\mathrm{Sp}(2g,\R)}
\newcommand{\slzz}{\mathrm{SL}(2g,\Z)}
\newcommand{\slzn}{\mathrm{SL}(n,\Z)}
\newcommand{\slrr}{\mathrm{SL}(2g,\R)}

\newcommand{\aut}{\textnormal{Aut}(\pi_1(S))}

\newcommand{\out}{\textnormal{Out}(\pi_1(S))}
\newcommand{\hrho}{\overline{\rho}}
\newcommand{\homrep}{\mathrm{H}_1(S,\Z)}

 {\left\lbrace\begin{array}{@{}l@{}}}%
 {\end{array}\right.}

{\begin{list}{}
         {\setlength{\leftmargin}{#1}}
         \item[]
}
{\end{list}}

\def\qr#1#2{%
      \raise1ex\hbox{$#1$}\Big/ \lower1ex\hbox{$#2$}%
}

\def\qrr#1#2{%
      \raise1ex\hbox{$#1$}\Big/\Big/ \lower1ex\hbox{$#2$}%
}

\def\ql#1#2{%
      \lower1ex\hbox{$#1$}\Big\backslash \raise1ex\hbox{$#2$}%
}

 {\left\lbrace\begin{array}{@{}l@{}}}%
 {\end{array}\right.}

\begin{document}
\title[Modular orbits on the representation spaces of compact abelian Lie groups]{Modular orbits on the representation spaces of compact abelian Lie groups}

\author[Y. Bouilly]{Yohann Bouilly}
\address{\textnormal{\textbf{Yohann Bouilly}} IRMA, UMR 7501 7 rue René-Descartes 67084 Strasbourg Cedex, France}
\email{bouilly@math.unistra.fr}

\author[G. Faraco]{Gianluca Faraco}
\address{\textnormal{\textbf{Gianluca Faraco}} Max Planck Institute for Mathematics, Bonn, Germany}
\address{\textnormal{Current address:} Mathematisches Institut Rheinische Friedrich-Wilhelms-Universit\"at, Bonn, Germany}
\email{gfaraco@uni-bonn.de}

\date{\today}
\dedicatory{}

\begin{abstract}
Let $S$ be a closed surface of genus $g$ greater than zero. In the present paper we study the topological-dynamical action of the mapping class group on the $\mathbb T^n$-character variety giving necessary and sufficient conditions for Mod$(S)$-orbits to be dense. As an application, such a characterisation provides a dynamical proof of the Kronecker's Theorem concerning inhomogeneous diophantine approximation.
\end{abstract}

\maketitle
\tableofcontents

\section{Introduction} \label{intro}

\noindent Let $S$ be a closed, connected and oriented topological surface of genus $g$. Its fundamental group $\pi_1(S)$ admits the presentation 
\begin{equation}
    \Big\langle\alpha_1,\beta_g,\dots,\alpha_g,\beta_g\mid \prod_{i=1}^g [\alpha_i,\beta_i]=1\Big\rangle.
\end{equation}

\noindent Let $G$ be a connected Lie group and let \textnormal{Hom}$(\pi_1(S), G)$ denote the set of representations of $\pi_1(S)$ in $G$. Such a set can be topologized with the compact-open topology and the resulting space is commonly known as \emph{representation space}.  There is a natural action of $G$ on \textnormal{Hom}$(\pi_1(S), G)$ obtained by post-composing representations with inner automorphisms of $G$. The resulting quotient space \textnormal{Hom}$(\pi_1(S), G)/G$ is a space canonically associated to $S$ (or $\pi_1(S)$) and $G$. When $G$ is an algebraic and reductive Lie group, the quotient space is commonly known as $G$-\emph{character variety of} $\pi_1(S)$. It can be shown that \textnormal{Hom}$(\pi_1(S),G)/G$ identifies with the moduli space of isomorphism classes of flat $G$-bundles over $S$.

\smallskip

\noindent We next consider the effect of changing the presentation of $\pi_1(S)$. This can be done by pre-composing any representation with an automorphism $\phi\in$ \textnormal{Aut}$(\pi_1(S))$ such that any representation $\rho$ is sent to $\rho\circ\phi^{-1}$. We can, therefore, consider the action of \textnormal{Aut}$(\pi_1(S))$ on \textnormal{Hom}$(\pi_1(S), G)$. The actions of $G$ and $\aut$ can be combined together in the following way: For any pair $(\phi,g)\in \aut\times G$ and $\gamma\in\pi_1(S)$ we define
\[ (\phi,g)\cdot \rho(\gamma)=g\big(\rho\circ\phi^{-1}(\gamma)\big)g^{-1} .
\] It is clear that the action of the normal subgroup \textnormal{Inn}$(\pi_1(S))<\aut$, consisting of inner automorphisms, is absorbed into the action of $G$. In other words, the action of \textnormal{Inn}$(\pi_1(S))$ on \textnormal{Hom}$(\pi_1(S), G)$ descends to the trivial action on \textnormal{Hom}$(\pi_1(S), G)/G$. Thus, there is a well-defined action of the \emph{outer automorphisms group} \textnormal{Out}$(\pi_1(S))=\aut/\textnormal{Inn}(\pi_1(S))$ on  \textnormal{Hom}$(\pi_1, G)\slash G$.\\

\noindent When $S$ is a closed surface, the group $\out$ has a very explicit geometric interpretation. Every homeomorphism of $S$ determines an automorphism of $\pi_1(S)$. On the other hand, for closed surfaces, the Dehn-Nielsen-Baer theorem \cite{NI27} states that every automorphism of $\pi_1(S)$ is induced by a homeomorphism of $S$ - this actually true also for the punctured torus, but it is no longer true for other surfaces with boundary. Now, homeomorphisms which are isotopic can be considered as equivalent, and determine conjugate automorphisms of $\pi_1(S)$, \emph{i.e.} a well-defined element of $\out$. Conversely, automorphisms which are conjugate can be considered as equivalent, and determine isotopic homeomorphisms. Therefore, for any closed surface the following isomorphism holds
\[ \frac{\text{Homeo}(S)}{\text{Isotopy}}\cong\out.
\] 

\begin{defn}
The \emph{mapping class group of $S$} is defined as Mod$(S)=\displaystyle\frac{\text{Homeo}(S)}{\text{Isotopy}}$.
\end{defn}

\noindent The character variety can be endowed with a symplectic structure which is preserved by the $\mathrm{Out}(\pi_1(S))$-action, see Goldman \cite{GO84}. By taking the volume form associated to the symplectic structure, we obtain a finite measure $\mu_S$ and the space \textnormal{Hom}$(\pi_1(S), G)\slash G$ turns into a measured space on which the group \textnormal{Out}$(\pi_1(S))$ acts preserving measure. The dynamic of this action is known for compact Lie groups:
\begin{thmnn}[Goldman, Pickrell-Xia]\label{thm3}
Let $G$ be a compact Lie group. Then the mapping class group acts ergodically on each connected components of \textnormal{\emph{Hom}}$(\pi_1(S),G)/G$ with respect to the finite measure $\mu_S$.
\end{thmnn}

\noindent A direct consequence of ergodicity is that almost every $\mathrm{Mod}(S)$-orbit is dense. A more subtle problem concerns the topological dynamics of the mapping class group action on the space \textnormal{Hom}$(\pi_1(S),G)/G$. The topological-dynamical problem is definitely more delicate since no longer we may ignore invariant subsets of measure zero. For instance, if $H$ is a finite subgroup of $G$, the space \textnormal{Hom}$(\pi_1(S),H)$ is finite and its image in \textnormal{Hom}$(\pi_1(S),G)/G$ is an invariant closed subset under the action of the mapping class group. It follows that, even when the action of Mod$(S)$ is ergodic, not all the orbits are dense in \textnormal{Hom}$(\pi_1(S),G)/G$. In \cite[Problem 2.7]{GO06}, Goldman posed the following problem - we refer to it in the sequel as main problem.

\smallskip

\begin{quote}
\begin{quote}
\textbf{Main Problem:} \emph{Determine necessary and sufficient conditions on a general representation $\rho$ for its orbit \textnormal{Mod}$(S)\cdot \rho$ to be dense.}
\end{quote}
\end{quote}

\smallskip

\noindent We introduce the following

\begin{defn}
A representation $\rho:\pi_1(S)\longrightarrow G$ is defined as \emph{dense representation} if the image of $\rho$ is dense in $G$.
\end{defn}

\noindent For closed surfaces, the following claim is expected: If the image of a representation $\rho$ is dense in $G$, then the Mod$(S)$-orbit of $\rho$ is dense in \textnormal{Hom}$(\pi_1(S),G)/G$. This is currently true for representations in SU$(2)$. In fact, this case has been completely treated by Previte and Xia \cite{PrX2} and it is based on their earlier work \cite{PrX1} on which they consider the case of representations in SU$(2)$ for the punctured torus. For all other compact Lie groups the problem is still open, and in the present work we provide a positive answer also in the case of the $n$-dimensional torus $\T^n$. The following Theorem is in fact the main result of this work.

\begin{thma}\label{thmbf}
Let $S$ be a surface of genus $g\ge1$ and $\pi_1(S)$ its fundamental group. Let $\rho:\pi_1(S)\longrightarrow \Bbb T^n$ be a representation. Then the image of $\rho$ is dense in $\Bbb T^n$ if and only if the mapping class group orbit \emph{Mod}$(S)\cdot\rho$ is dense in the representation space.
\end{thma}

\noindent The proof strongly relies on the explicit knowledge of the objects involved. In fact, the $n$-torus has a well-known description and, thanks to the abelian property, the character variety coincides with the representation space since the action of $\Bbb T^n$ by conjugation is trivial. Even better, the representation space can be identified with a torus of suitable dimension, hence the description of the representation space $-$ and then of the character variety $-$ is very explicit. The main difficulties in the abelian case concern questions coming from number theory and ergodic theory, see Section \S\ref{cwkat}.

\smallskip

\subsection{Strategy of the proof and related results} Each given representation $\rho:\pi_1(S)\longrightarrow \T^n$ induces a homological representation $\hrho:\homrep\longrightarrow \T^n$, namely an element of the \emph{homological representation space} $\textnormal{Hom}\big(\homrep, \T^n\big)$. The map associating to any representation $\rho$ its homological representation defines a bijection between the representation space and the homological representation space. This essentially follows because the commutator group $[\pi_1(S),\pi_1(S)]$ is trivially a subgroup of $\textnormal{ker}(\rho)$ in the abelian case, and such a property is no longer true for a generic non-abelian Lie group. There is also a well-defined action of the symplectic group $\spz$ on the space $\textnormal{Hom}\big(\homrep, \T^n\big)$ by precomposition. Given a representation $\rho:\pi_1(S)\longrightarrow \T^n$ and its induced representation $\hrho:\homrep\longrightarrow \T^n$, the Mod$(S)$-orbit of $\rho$ coincides with the $\spz$-orbit of $\hrho$. As an immediate consequence we obtain an equivalent version of the main Theorem \ref{thmbf}, namely we have the following.

\begin{thma}\label{thmbf2}
Let $S$ be a surface of genus $g\ge1$ and let $\hrho:\homrep\longrightarrow \Bbb T^n$ be a representation. Then the image of $\rho$ is dense in $\Bbb T^n$ if and only if the symplectic group orbit $\spz\cdot\hrho$ is dense in the homological representation space.
\end{thma}

\noindent Along the way of our investigation we shall remark the following result:
\begin{quote}
 \emph{The action of the Torelli group $\mathcal{I}(S)$ on the representation space $\emph{\textnormal{Hom}}\big(\pi_1(S), \T^n\big)$ is trivial.}
 \end{quote}
\noindent  For a proof of this fact we refer to our Proposition \ref{torga} below. This claim seems to be known, however, at the best knowledge of the authors, it has never been stated explicitly. In contrast to the abelian case, the first named author has shown in \cite{YB} that for any connected, compact and semi-simple Lie group, the action of the Torelli group on the character variety is ergodic.

\smallskip

\noindent Given a representation, we have reduced the problem to the study of $\spz$-orbit of instead of modular orbits. This makes the study of orbits more understandable because the symplectic group is linear. We shall make the action even more explicit by identifying a representation with a matrix in the space $\text{M}\big(n,2g;\,\T\big)$. Such a space will be introduced later on in section \S\ref{matrep}. After these reductions, we shall see that we are in the position to apply Ratner's Theorem for studying orbit closures. In particular, we shall derive our main Theorem \ref{thmbf}. 

\begin{rmk} For the torus, the reader may notice that Theorems \ref{thmbf} and \ref{thmbf2} are not only equivalent but actually the same statement in the strict sense. Indeed, in this very particular case the following equalities $\pi_1(S)=\homrep$ and Mod$(S)=\text{SL}(2,\Z)=\spzz$ holds.
\end{rmk}

\noindent The strategy we propose for Theorem \ref{thmbf} is different to the one developed by Previte-Xia to show their main theorem \cite[Theorem 1.4]{PrX2}. Let us briefly give some more details. Given a dense representation $\rho:\pi_1(S)\longrightarrow \text{SU}(2)$ - Previte-Xia defined such a representation \emph{generic} (see \cite[Definition 1.6]{PrX2}) - they firstly found a handle $\Sigma$, namely a one-holed torus, such that the restriction of $\rho$ to $\pi_1\Sigma$ is dense. After obtaining a dense handle, they proceed to demonstrate the base density theorem for the $(n+2g- 2)$-holed torus. A similar process in the abelian case is not possible because dense handles do not always exist, see the discussion at Section \S\ref{fdv} in Appendix \S\ref{mtoht}. In the light of Proposition \ref{torga}, we shall bypass this issue by looking at the $\spz$-action on the representation space as described above. 

\subsection{Connection with the Kronecker's Approximation Theorem}\label{cwkat} The dynamical result provided by Theorem \ref{thmbf} finds an application on the theory of geometry of numbers. An important theorem in this topic is the Kronecker’s theorem  concerning inhomogeneous Diophantine approximation, see Section \S\ref{apres} below for the precise statement.

\smallskip

\noindent By fixing a presentation of $\pi_1(S)$, we can associate to any representation $\rho:\pi_1(S)\longrightarrow\T^n$ a matrix $\Theta_\rho\in\text{M}\big(n,2g;\R\big)$ - see Definition \ref{mr} and \ref{matrep} below for the details. We shall prove that a representation $\rho$ is dense if and only if the rows of the matrix $\Theta_\rho$ satisfy the hypothesis of Kronecker's theorem, this is our Theorem \ref{codr}. On the other hand, our main result says that the modular orbit of a representation $\rho$ is dense in the representation space if and only if $\rho$ itself is a dense representation. As the representation space identifies with $\T^{2ng}$, see section \S\ref{tcv} below, Theorem \ref{thmbf2} provides a dynamical proof of Kronecker's Theorem in the cases of $l=m=2g$, for some $g\ge1$. 
More precisely we show that:

\begin{thma}\label{thmbf4}
Let $m=2g$ with $g\ge1$. Let $b^{(i)}=\big(b_1^{(i)},\dots,b_m^{(i)}\big)$, with $i=1,\dots,n$, be vectors of $\R^m$ such that $b^{(1)},\dots,b^{(n)},$ $\pi e_1,\dots,\pi e_m$ are linearly independent over $\Q$ in the vector space $\R^m$. Let $A\in\text{M}\big(n,m;\R)$ be a real matrix and let $\varepsilon$ be a positive number. Then there is an element $K\in\spz$ such that
\begin{equation}
 \Big| \Big|\, A-B\, K\,\Big|\Big|<C\varepsilon\textnormal{ mod } 2\pi.
\end{equation} where $C$ is a constant depending only on $m$ and $n$ and the norm is any norm on $\text{M}\big(n,m;\R)$.
\end{thma}

\noindent This is a sharper simultaneous approximation result because, in principle, one can always find a matrix $K\in\text{M}\big(2g,\Z\big)$ according to Kronecker's Theorem. For the sake of comprehension, the proof of Theorem \ref{thmbf4} is delayed until section \S\ref{apres} when we have developed the theory and notation even further. As we shall see, the proof of this result reduces to prove the following characterisation.

\begin{propa}\label{aiffd}
Theorem \ref{thmbf} holds if and only if Theorem \ref{thmbf4} holds.\\
\end{propa} 

\subsection{Related dynamical problems}\label{rdp} For a generic compact Lie groups $G$, the main issues one has to face are mainly two. First of all, the group $G$ may not have a nice description. In fact, among all compact Lie groups we found the classical simple Lie groups belonging to the four families SU$(n+1)$, SO$(2n+1)$, Sp$(n)$, SO$(2n)$, but also the five exceptional Lie groups corresponding to the Dynkin diagrams G$_2$, F$_4$, E$_6$, E$_7$, E$_8$ which are harder to treat. The second issue comes from the fact that most of the representation spaces and their quotients \textnormal{Hom}$(\pi_1(S),G)/G$ have not an explicit description to work with. The abelian case is not the only one on which these issues vanish. Also in the case of SU$(2)$ they completely miss since both the group and the representation space were already well-known in the literature. For open surfaces of positive genus and positive number of boundary components, there is a further problem to be addressed. Suppose $S$ has genus $g$ with boundary of $k$ disjoint circles. A relative character variety is a slice of the space \textnormal{Hom}$(\pi_1(S),G)/G$ subject to the condition imposed by some finite collection of $k$ conjugacy classes. Like in the close case, the relative character variety carries a symplectic structure which is preserved by the mapping class group action. In \cite{PX2}, Pickrell-Xia established the ergodicity of the mapping class group action with respect to the symplectic measure for $k>2$. Their result may be seen as the follow up of the Goldman's work \cite{GO97}, where he considered the cases of groups whose simple factors are locally isomorphic to SU$(2)$. Very recently, in \cite{GLX}, Goldman-Lawton-Xia announced a proof in the case of SU$(3)$ based on different techniques than in \cite{PX2}.\\
\noindent As for the modular orbits characterisation, the main problem posed by Goldman has been answered only in the case of SU$(2)$, as mentioned above. This work is therefore a new and partial development of the wider program to understanding the dynamics of the mapping class group. In \cite{BKMS}, Biswas-Koberda-Mj-Santharoubane have consider the opposite problem of characterising representations having \emph{finite} modular orbit. For any fixed Lie group $G$, they have showed that any representation, with values in $G$, having finite modular orbit has necessarily finite image in $G$. The case $G=\slc$ has been handled also by Biswas-Gupta-Mj-Whang in \cite{BGMW}.

\smallskip

\noindent The non-compact case is even more complicated and delicate; let us spend a few words. The current situation is different for non-compact Lie groups and we do not expect an analog theorem. For compact Lie groups $G$, the space \textnormal{Hom}$(\pi_1(S),G)/G$ has non-trivial homotopy type, and Theorem \ref{thm3} says that the dynamics of the action of the mapping class is chaotic on each connected component. On the other hand, when $G$ is a non-compact semisimple Lie group, the space \textnormal{Hom}$(\pi_1(S),G)/G$ contains open contractible components on which the action of the mapping class group is properly discontinuous. Often, these components correspond to locally homogeneous structures uniformizing $S$. A remarkable case is that of $\pslr$. It is well-known that the $\pslr$-character variety has $4g-3$ connected components indexed by the Euler class, taking values in a finite set of $\mathbb{Z}$, where $g$ denotes the genus of $S$, see \cite{GO88}. Two of these components correspond to the Teichm\"uller spaces $\mathcal{T}(S)$ and $\mathcal{T}(\overline{S})$ of $S$ (where $\overline{S}$ is the surface $S$ taken with the opposite orientation). The action of Mod$(S)$ is known to be proper on these components and it is conjectured to be ergodic on the others. This conjecture is currently treated in the case of genus $2$ surfaces thanks to recent results of March\'e-Wolff. They proved the conjecture about ergodicity is true for Euler class equal to $\pm 1$ and decompose the connected component of Euler class $0$ in two subspace on which the mapping class group acts ergodically \cite{MW2, MW}. 
As it may easy to expect, even less is known about the topological dynamics of the mapping class group on the $\pslr$-character variety. In the case of genus two, one among of the consequences of March\'e and Wolff's result is the following claim: in each subspace of the character variety on which the action of the mapping class group is ergodic there is a full measure subset of representations whose mapping class group orbit is dense in this subspace. Like in the compact case, we can pose the following question: Does a dense representation $\rho:\pi_1(S)\to\pslr$ have dense Mod$(S)$-orbit? Answering to this question is even more tricky and for surfaces with boundary we already know counterexamples - see \cite{PX3} for an example in $\slc$. 

\subsection{Structure of the paper} The paper is organised as follow. In Section \S\ref{tcv} we begin with a description of the $\T^n$-character variety and then subsequently introduce the homological representation space and show the identification with the character variety. We finally describe the action of the symplectic group $\spz$ on the homological representation space. As a consequence, we shall derive the Proposition \ref{torga} and the equivalence of Theorems \ref{thmbf} and \ref{thmbf2}.  In Section \S\ref{dr} we shall give a complete characterisation of dense representations in the $n$-dimensional torus by proving Theorem \ref{codr}. In section \S\ref{ratthm} we shall finally derive our main Theorem \ref{thmbf}. In the last section, we prove Proposition \ref{aiffd} and indeed Theorem \ref{thmbf4} establishing the connection of our dynamical result with the Kronecker's Approximation Theorem. We finally conclude with a serie of appendix on which we shall discuss some further aspects related to our project. Appendix \S\ref{mtoht} we discuss about a direct approach to our problem which works for a fairly general class of representations. In Appendix \S\ref{swpc} we digress a little by providing a brief description of the relative $\T^n$-character variety for surfaces with one puncture and then we claim that our main results extend to one-punctured surfaces.

\subsection*{Acknowledgment} The first named author would like to thanks Olivier Guichard for his comments and guidance. The second named author would like to thank Anish Ghosh for introducing him in the nice theory of groups actions on homogeneous spaces and Ratner's theory. Part of this of work was carry out during the visiting of second named author to the University of Bologna in November 2019 and TIFR Mumbai in December 2019. He is grateful for the nice hospitality in both places. Both the authors wish to thank Subhojoy Gupta, Fanny Kassel and Maxime Wolff for their interest on this work. Both the authors finally wish to thank the organisers of the \emph{Conference on Geometric Structures in Nice} which took place in January 2019 where their collaboration began. Finally, we thank an anonymous referee for the helpful suggestions given to improve the presentation of this paper.

\section{$\Bbb T^n$-character variety}\label{tcv}

\noindent In this work we are interested in characterising the orbits of the Mod$(S)$-action on \textnormal{Hom}$(\pi_1(S),G)/G$ where $G$ is a compact, connected  and abelian Lie group. It is classical to see that any such a group is isomorphic to $\Bbb T^n$, the $n$-dimensional torus for some positive $n$ see for instance \cite[Corollary 3.7]{BtD}. The specific interest for the abelian case comes from its connection with abstract harmonic analysis, the geometry of numbers and the theory of group actions on homogeneous spaces (connections with Ratner's Theorem, see section \S\ref{ratthm}).

\medskip

\noindent In the introduction we have given a very brief view of the character variety for a generic compact Lie group $G$. In this section we specialise the discussion for compact and connected abelian Lie groups. From the Lie theory, any such a group is known to be a $n$-dimensional torus, namely the product of $n$ copies of the unit circle $\Bbb S^1$. In the present work $\Bbb S^1$ is seen as $\big\{ e^{i\theta}\,|\, \theta\in[0,2\pi)\big\}$ where $[0,2\pi)$ carries the quotient topology obtained identifying the boundary points of the closed interval $[0,2\pi]$. Consequently, the $n$-torus $\T^n$ is defined as $\big\{\big(e^{i\theta_1},\dots, e^{i\theta_n}\big)\,|\, \theta_i\in [0,2\pi), \text{ for any } i=1,\dots,n\big\}$ endowed with the product topology.

\smallskip

\noindent Let $S$ be a closed surface and let $\alpha_1,\beta_1,\dots,\alpha_g,\beta_g$ be any standard generating system of the fundamental group. The choice of a representation $\rho:\pi_1(S)\longrightarrow \T^n$ amounts to choose for each generator an element of $\T^n$ such that these elements satisfy the condition imposed by the presentation of the fundamental group of $S$. Since $\T^n$ is an abelian group, the relation $[A_1,\,B_1]\cdots[A_g,\,B_g]=1$ is automatically satisfied for any choice of $2g$ elements in $(A_1,B_1,\dots,A_g,B_g)\in\T^n$. Thus, the representation space can be identified with the full group $\big(\T^n\big)^{2g}\cong\T^{2ng}$. Even more, thanks again to the abelian property, the action of $\T^n$ on \textnormal{Hom}$\big(\pi_1(S),\T^n\big)$ by post-composition with inner automorphisms of $\T^n$ is trivial. As a consequence, the $\T^n$-character variety coincides with the representation space.


\subsection{Homological representations} Let $\mathrm{H}_1(S,\Z)$ be the first homology group. The close connection between the objects $\pi_1(S)$ and $\mathrm{H}_1(S,\Z)$ is well-known, indeed the latter is known to be isomorphic to the abelianization of $\pi_1(S)$. As we have seen above, the representation space $\textnormal{Hom}\big(\pi_1(S),\T^n\big)$ naturally identifies with the $2gn$-dimensional torus assigning to any representation $\rho$ the $2g$-tuple $\big(\rho(\alpha_1),\rho(\beta_1),\dots,\rho(\alpha_g),\rho(\beta_g)\big)$, where $\alpha_1,\beta_1,\dots,\alpha_g,\beta_g$ is a basis for $\pi_1(S)$. Every representation $\rho$ fails to be injective and its kernel \textnormal{ker}$(\rho)$ always contains the subgroup generated by the commutators since the target is abelian. Therefore, $\rho$ boils down to a representation
\[ \hrho:\homrep\cong\frac{\pi_1(S)}{\big[\pi_1(S),\pi_1(S)\big]}\longrightarrow \T^n,  \qquad \hrho\big(\big[\gamma\big]\big):=\rho(\gamma).
\]
\noindent In fact, let $\gamma\in\pi_1(S)$ and let $\big[\gamma\big]$ be its image via the canonical projection $p:\pi_1(S)\longrightarrow \mathrm{H}_1(S,\Z)$. Let $\gamma+\big[\sigma_1,\sigma_2\big]$ be a representative of $\big[\gamma\big]$. Since the following chain of equalities holds
\[ \rho\big(\gamma+\big[\sigma_1,\sigma_2\big]\big)=\rho\big(\gamma\big)\rho\big(\big[\sigma_1,\sigma_2\big]\big)=\rho\big(\gamma\big),
\] the representation $\hrho$ is well-defined and the image does not depend on the choice of the representative. Furthermore, the image of $\rho$ agrees with the image $\hrho$ by contruction.

\begin{defn}
We define the \emph{homological representation space} as the set $\textnormal{Hom}\big(\mathrm{H}_1(S,\Z),\T^n\big)$ of representations of $\mathrm{H}_1(S,\Z)$ in $\T^n$ endowed with the compact-open topology.
\end{defn}

\begin{lem}
The homological representation space $\emph{\textnormal{Hom}}\big(\mathrm{H}_1(S,\Z),\T^n\big)$ identifies with the $2gn$-dimensional torus $\T^{2gn}$.
\end{lem}

\begin{proof}
To any representation $\hrho$ we can assign the $2g$-tuple defined as $\big(\hrho([\alpha_1]),\hrho([\beta_1]),\dots,\hrho([\alpha_g]),\hrho([\beta_g])\big)$, where the collection $\big[\alpha_i\big],\big[\beta_i\big]$, $1\le i\le g$ is a fixed basis of the homology group $\mathrm{H}_1(S,\Z)$. Conversely, since $\Bbb T^{2gn}$ is an abelian group, for any $2g$-tuple of $\big(\T^{n}\big)^{2g}$, say $\big(v_1,w_1,\dots,v_g,w_g\big)$, the universal property of free abelian groups implies the existence of a unique group homomorphism from $\homrep$ into the $n$-torus $\T^n$ which sends $[\alpha_i]$ to $v_i$ and $[\beta_i]$ to $w_i$, for every $i=1,\dots,g$.
\end{proof}

\noindent The implications of this lemma are quite simple, but of crucial importance.  Upon choosing a basis for $\pi_1(S)$; the representation space $\textnormal{Hom}\big(\pi_1(S),\T^n\big)$ identifies with the homological representation space $\textnormal{Hom}\big(\mathrm{H}_1(S,\Z),\T^n\big)$ and the identification is explicitely given by the association $\rho\mapsto\hrho$. According to this property, we derive the following lemma.

\begin{lem}\label{reqhr}
Let $\rho_1,\rho_2:\pi_1(S)\longrightarrow\T^n$ be two representations. Then $\rho_1\equiv\rho_2$ if and only if $\hrho_1\equiv\hrho_2$.
\end{lem}

\begin{proof}
This is just a matter of definitions given so far. The necessary condition follows trivially. The sufficient condition follows from $\hrho\big(\big[\gamma\big]\big)=\rho(\gamma)$ for any $\gamma\in\pi_1(S)$. \qedhere
\end{proof}

\smallskip

\subsection{Actions of the symplectic group $\spz$} In this section we are going to describe the action of the symplectic group $\spz$ both on the representation space and on the homological representation space. 

\subsubsection{The symplectic group $\spz$.} We begin with recalling some standard notions. The algebraic intersection number 
\[ \cap: \mathrm{H}_1(S,\Z)\times \mathrm{H}_1(S,\Z)\longrightarrow \Z
\] extends uniquely to a nondegenerate, alternating bilinear map
\[ \cap: \mathrm{H}_1(S,\R)\times \mathrm{H}_1(S,\R)\longrightarrow \R
\] which realises $\mathrm{H}_1(S,\R)$ as a symplectic vector space. 

\begin{defn}
A collection of elements $\big[\alpha_i\big],\big[\beta_i\big]$, $1\le i\le g$ of $\mathrm{H}_1(S,\Z)< \mathrm{H}_1(S,\R)$ such that
\[ \big[\alpha_i\big]\cap \big[\alpha_j\big]=\big[\beta_i\big]\cap \big[\beta_j\big]=0, \quad \big[\alpha_i\big]\cap \big[\beta_j\big]=\delta_{ij}
\] for all $i,j$ with $1\le i,j\le g$ is called a \emph{symplectic basis of the group} $\mathrm{H}_1(S,\Z)$ or a basis for the symplectic vector space $\big(\mathrm{H}_1(S,\Z),\,\cap\,\big)$. We define a collection of curves $\alpha_i,\beta_i$ such that $\big\{[\alpha_i],[\beta_i]\big\}$ is a symplectic basis as \emph{geometric symplectic basis} for $\pi_1(S)$.
\end{defn}

\noindent The matrix associated to the antisymmetric bilinear form $\cap$ on the basis $\big[\alpha_i\big],\big[\beta_i\big]$ is the $2g\times 2g$ blockwise diagonal matrix
\[ J=
\begin{pmatrix}
J_o & & \\
 & \ddots & \\
  & & J_o
\end{pmatrix} \quad \text{ with } \quad J_o=
\begin{pmatrix}
0 & 1\\
-1 & 0
\end{pmatrix}.\\
\] 
\text{}\\
\noindent The symplectic linear group Sp$(2g,\R)$ is defined as the group of invertible matrices $A$ satisfying the relation $AJA^{t}=J$ and we denote by $\spz$ the subgroup of those matrices with integer coefficients.

\begin{rmk}\label{specsubgroup}
Here, the symplectic group $\spr$ is the subgroup of $\slrr$ of matrices preserving the alternating $2$-form $\omega=e_1\wedge e_2+\cdots+e_{2g-1}\wedge e_{2g}$. $\spr$ contains the $g$-times product $\slr\times \cdots\times \slr$ as a proper subgroup. In turns, the group $\spz$ contains the $g$-times product $\slz\times\cdots\times\slz$ as a proper subgroup. This property will be useful in the sequel, see Appendix \S\ref{mtoht}.
\end{rmk}

\noindent An orientation preserving homeomorphism induces an isomorphism in homology which preserves the intersection form $\cap$ defined above. Since isotopic homeomorphisms induce the same map in homology, there is a representation 
\[ \mu:\text{Mod}(S)\longrightarrow \text{Aut}^+\big(\mathrm{H}_1(S,\Z)\big)\cong \text{SL}(2g,\Z).
\]  As each homeomorphism preserves the intersection form $\cap$, the image of $\mu$ lies also inside Sp$(2g,\R)$. Therefore the image of $\mu$ lies inside SL$(2g,\Z)$ $\cap$ Sp$(2g,\R)=\spz$. The representation $\mu:\text{Mod}(S)\longrightarrow\spz$ $-$ usually called symplectic representation of Mod$(S)$ $-$ is surjective with kernel $\mathcal{I}(S)$. The subgroup $\mathcal{I}(S)$ is called \emph{the Torelli subgroup of }Mod$(S)$.

\begin{rmk}
In the genus one case the Torelli subgroup is trivial. Indeed, Mod$(T)\cong\slz\cong \spzz$.
\end{rmk}


\subsubsection{Comparison of the \textnormal{Mod}$(S)$-orbits with the $\spz$-orbits.} 
We now consider the effect of changing the basis of $\mathrm{H}_1(S,\Z)$ pre-composing any homological representation with an automorphism $\phi\in\textnormal{Aut}^+\big(\mathrm{H}_1(S,\Z)\big)$ such that any representation $\hrho$ is sent to $\hrho\circ\phi^{-1}$. We can, therefore, consider $\slzz$-action on the space $\textnormal{Hom}\big(\mathrm{H}_1(S,\Z),\T^n\big)$. Of course, this action restricts to an action of the symplectic group $\spz$ on the same space. We are interested in studying the $\spz$-orbits in the homological representation space. The main goal of this section is proving the following claim.

\begin{prop}\label{orbcomp}
Let $\rho_1,\rho_2:\pi_1(S)\longrightarrow \T^n$ be two representations and let $\hrho_1,\hrho_2:\mathrm{H}_1(S,\Z)\longrightarrow \T^n$ be the induced representations. Suppose there is $\phi\in \text{\emph{Mod}}(S)$ such that $\rho_2=\rho_1\circ\phi$. Then $\hrho_2=\hrho_1\circ\mu(\phi)$, where $\mu$ is the symplectic representation of \emph{Mod}$(S)$.
\end{prop}

\begin{proof}
Let $\phi:\pi_1(S) \longrightarrow \pi_1(S)$ be any element of $\out$. As the image of any commutator is also a commutator, the mapping $\phi$ boils down to a isomorphism in homology $\mu(\phi):\homrep\longrightarrow \homrep$. Two mappings $\phi_1$ and $\phi_2$ boil down to the same isomorphism in homology if and only if $\phi_2\circ\phi_1^{-1}$ descends to the identity map in homology, that is $\phi_2\circ\phi_1^{-1}$ is an element of the Torelli subgroup by a Theorem of Johnson, \cite{JD}. Therefore, the association $\phi\longmapsto \mu(\phi)$ defines the symplectic representation $\mu$ seen above. Look at now the following commutative diagram

\begin{equation}\label{diag}
\begin{tikzcd}
  \mathrm{H}_1(S,\Z) \arrow{rr}[]{\hrho_2=\overline{\rho_1\circ\phi}}
    &   & \T^n\\
  \pi_1(S) \arrow{u}{p} \arrow{d}{p} \arrow{r}{\phi}
& \pi_1(S)  \arrow{d} \arrow{r}{\rho_1} & \T^n \arrow{u}{\text{id}}\arrow{d}{\text{id}}\arrow{d}\\
\mathrm{H}_1(S,\Z) \arrow{r}[]{\mu(\phi)}
    & \homrep \arrow{r}[]{\hrho_1}  & \T^n
\end{tikzcd}
\end{equation}

\text{}\\
where $p$ is the canonical projection. As $\rho_2=\rho_1\circ \phi$ by assumption, it turns out $\hrho_2=\overline{\rho_1\circ\phi}=\hrho_1\circ\mu(\phi)$ as desired.\qedhere
\end{proof}


\subsubsection{Direct consequences.} Proposition \ref{orbcomp} leads to some interesting consequences that we are going to show. The first one concerns the action of the Torelli subgroup $\mathcal{I}(S)$ on the representation space $\textnormal{Hom}\big(\pi_1(S),\T^n\big)$.

\begin{prop}\label{torga}
The action of the Torelli group $\mathcal{I}(S)$ on the representation space $\textnormal{\emph{Hom}}\big(\pi_1(S), \T^n\big)$ is trivial.
\end{prop}

\begin{proof}
Let $\rho_1\in \textnormal{Hom}\big(\pi_1(S), \T^n\big)$ be any representation and let $\phi\in\mathcal{I}(S)$. Set $\rho_2=\phi\cdot\rho_1=\rho_1\circ\phi^{-1}$. Proposition \ref{orbcomp} implies that $\hrho_1=\hrho_2$ because $\mu(\phi)=1$. We now invoke Lemma \ref{reqhr} to conclude $\rho_1=\rho_2$, namely the action of $\phi$ is trivial.
\end{proof}

\begin{rmk}
An alternative argument is the following. Let $\gamma$ and $\gamma'$ two isotopy classes of simple closed non separating curves. In \cite{JD}, Johnson noticed that $\gamma$ and $\gamma'$ are $\mathcal{I}(S)$-equivalent if and only if they represent the same element in $\mathrm{H}_1(S,\Z)$. Fix a basis for $\pi_1(S)$ of simple closed non-separating curves, and let $\rho_1,\rho_2$ two $\mathcal{I}(S)$-equivalent representations, that is $\rho_2=\rho_1\circ\phi$. For any generator $\gamma$ we have the following chain of equalities
\[ \rho_2\big(\gamma\big)=\hrho_2\big([\gamma]\big)=\overline{\rho_1\circ\phi^{-1}}\big([\gamma]\big)=\hrho_1\big(\phi^{-1}([\gamma])\big)=\hrho_1\big([\gamma]\big)=\rho_1(\gamma),
\] that imply $\rho_1=\rho_2$ as desired.
\end{rmk}

\noindent The $n$-torus $\T^n$ is a compact and connected Lie group and hence mapping class group Mod$(S)$ acts ergodically on the representation space - see Theorem \ref{thm3} in the introduction. As the action of the Torelli subgroup $\mathcal{I}(S)$ is trivial, the action of the quotient group is also well-defined and the following holds.

\begin{prop}\label{symacts}
The action of $\spz\cong\displaystyle\frac{\textnormal{Mod}(S)}{\mathcal{I}(S)}$ on $\textnormal{\emph{Hom}}\big(\pi_1(S), \T^n\big)$ is ergodic with respect to the finite measure $\mu_S$.
\end{prop}

\noindent As the homological representation spaces identifies with the representation space, it also carries a finite measure. Calling $\imath$ the identifying map, this finite measure can be seen as the pull-back measure $\imath^*\mu_S$, where $\mu_S$ is the finite measure carried by the representation space. 

\begin{cor}
The action of the symplectic group $\spz$ on the space $\textnormal{\emph{Hom}}\big(\homrep, \T^n\big)$ is ergodic with respect to the finite measure $\imath^*\mu_S$.
\end{cor}

\noindent As a final consequence we have the following characterisation.

\begin{prop}
Let $\rho:\pi_1(S)\longrightarrow\T^n$ be a representation and let $\hrho:\homrep\longrightarrow \T^n$ the homological representation induced by $\rho$. Then the mapping class group orbit \textnormal{Mod}$(S)\cdot\rho$ is dense if and only if the symplectic group orbit $\spz\cdot\hrho$ is dense.
\end{prop}

\begin{proof}
Proposition \ref{orbcomp} implies that the mapping class group orbit of $\rho$ coincides with the symplectic group orbit of $\hrho$ via the identification $\rho\mapsto\hrho$. Therefore, one orbit is dense if and only if the other is dense.
\end{proof}

\begin{cor}\label{aiffb}
Let $S$ be a surface of genus $g\ge1$. Then Theorem \ref{thmbf} holds if and only if Theorem \ref{thmbf2} holds.\\
\end{cor}

\subsection{The matrix presentation}\label{matrep} The $n$-torus $\T^n$ is also seen as the quotient of $\R^n$ by the action of the lattice $2\pi\Z^n$, indeed the exponential map provides an identification between $\R^n/2\pi\Z^n$ and the $n$-torus described above. We shall define the map
\begin{equation}\label{projexp}
 \textnormal{exp}: \R^n\longrightarrow \T^n \qquad \big(\theta_1,\dots,\theta_n\big)\longmapsto \big(e^{i\theta_1},\dots, e^{i\theta_n}\big)
\end{equation}
\noindent  as the \emph{canonical projection}. In the sequel it turns out also useful to look at the $n$-torus as the quotient of $\R^n$ with a suitable lattice $\Lambda=g\cdot \big(2\pi\Z^n\big)$ where $g\in \text{SL}(n,\Z)$. The reason of that will be discussed afterwards. We define $2\pi\Z^n$ as the \emph{the standard lattice} $-$ notice that this lattice is $2\pi$ times the usual standard lattice.

\smallskip

\noindent Fix a set of generators $\big\{\alpha_1,\beta_1,\dots,\alpha_g,\beta_g\big\}$ and let us consider $\T^n$ as the quotient of $\R^n$ by the action of the standard lattice. Let $\rho:\pi_1(S)\longrightarrow \T^n$ be any representations and set
\[ \rho(\alpha_i)=\big(e^{i\theta_{1,2i-1}},\dots,e^{i\theta_{n,2i-1}}\big)
\]
\[ \rho(\beta_i)=\big(e^{i\theta_{1,2i}},\dots,e^{i\theta_{n,2i}}\big)
\] for any $i=1,\dots,n$. The elements $\rho(\alpha_1),\rho(\beta_1),\dots,\rho(\alpha_g),\rho(\beta_g)$ generate the image of the representation $\rho$. Any generic element $\gamma\in\pi_1(S)$ may be seen as a word in the letters $\alpha_i,\beta_i$ for $i=1,\dots, 2g$. Hence, since $\Bbb T^{2gn}$ is abelian,
\begin{equation}
    \rho(\gamma)=\rho\big(w(\alpha_1,\beta_1,\dots\alpha_g,\beta_g)\big)=\rho(\alpha_1)^{k_1} \cdots \rho(\beta_g)^{k_{2g}}
\end{equation}
\noindent for some $k_1,\dots,k_{2g}\in \Z$. In particular, the element $\rho(\gamma)$ can be computed with the following matrix multiplication
\[
\begin{pmatrix}
\theta_{1,1} & \cdots & \theta_{1,i} & \cdots & \theta_{1,2g}\\
\vdots & & & & \vdots\\
\theta_{n,1} & \cdots & \theta_{n,i} & \cdots & \theta_{n,2g}
\end{pmatrix}
\begin{pmatrix}
k_1\\
 \vdots\\
 k_i\\
 \vdots\\
 k_{2g}
\end{pmatrix}.
\]

\begin{defn}\label{mr}
Let $\Theta_\rho$ be the matrix having as entries the values $\theta_{i,j}\in[0,2\pi)$ with $i=1,\dots,n$ and $j=1,\dots,2g$. We define $\Theta_\rho$ as \emph{the matrix associated to $\rho$} with respect the basis $\big\{\alpha_1,\beta_1,\dots,\alpha_g,\beta_g\big\}$ and the standard lattice $2\pi\Z^n$.
\end{defn}

\noindent In what follows, we shall often identify a representation $\rho$ with its associated matrix. Let us briefly see the reason why we are legitimated to do that. Consider the topological vector space $\text{M}\big(n,2g;\R\big)$ and introduce an equivalence relation where $A\sim B$ if and only if $A-B=2\pi H\in\text{M}\big(n,2g;2\pi \Z\big)$. The mapping $\imath$ associating to any $\rho$ its associated matrix $\Theta_\rho$ provides an homeomorphism between the representation space \textnormal{Hom}$\big(\pi_1(S),\T^n\big)$ and the quotient space $\text{M}\big(n,2g,\T\big)$. Moreover, the post-composition of the mapping $\rho\mapsto \overline{\rho}$ with $\imath^{-1}$ defines a homeomorphism between the spaces \textnormal{Hom}$\big(\mathrm{H}_1(S,\Z),\T^n\big)$ and $\text{M}\big(n,2g,\T\big)$.\\

\noindent Given a representation $\rho$, the matrix $\Theta_\rho$ depends on the choice of a set of generators for $\pi_1(S)$ and also on the choice of a lattice $\Lambda<\R^n$. Let us see how these choices affect definition \ref{mr}. We begin describing the effect of changing the set of generators of $\pi_1(S)$. 

\subsubsection{The effect of changing basis.}\label{ecb}  Given two basis $\mathcal{B}=\big\{\alpha_1,\beta_1,\dots,\alpha_g,\beta_g \big\}$ and $\mathcal{B}'=\big\{\alpha'_1,\beta'_1,\dots,\alpha'_g,\beta'_g \big\}$ of $\pi_1(S)$, we define $\Theta_\rho$ and $\Theta'_\rho$ the matrices associated to $\rho:\pi_1(S)\longrightarrow\T^n$ with respect to $\mathcal{B}$ and $\mathcal{B}'$ respectively. Every generator $\alpha'_l$ and $\beta_l'$ is a finite word in the letters $\alpha_1,\beta_1,\dots,\alpha_g,\beta_g$, so there are integers $a_{ij}$ with $i,j\in\big\{1,\dots,2g\big\}$ such that 
\[\rho(\alpha'_l)=\rho(\alpha_1)^{a_{2l-1\,1}} \cdots \rho(\beta_g)^{a_{2l-1\,2g}} \quad \text{ and }  \quad \rho(\beta'_l)=\rho(\alpha_1)^{a_{2l\,1}} \cdots \rho(\beta_g)^{a_{2l\,2g}}.
\]\\ Setting $A$ as the integral matrix $\big(a_{ij}\big)$ with $i,j\in\big\{1,\dots,2g\big\}$, a direct computation shows that $\Theta'_{\rho}$ equals $\Theta_\rho\cdot A$. Likewise, $\alpha_l,\beta_l$ are also finite words in the letters $\alpha'_1,\beta'_1,\dots,\alpha'_g,\beta'_g$. Hence, there exist integers $b_{ij}$  such that 
\[\rho(\alpha_l)=\rho(\alpha'_1)^{b_{2l-1\,1}} \cdots \rho(\beta'_g)^{b_{2l-1\,2g}} \quad \text{ and } \quad \rho(\beta_l)=\rho(\alpha'_1)^{b_{2l\,1}} \cdots \rho(\beta'_g)^{b_{2l\,2g}}.
\] Setting $B$ as the integral matrix $\big(b_{ij}\big)$ with $i,j\in\big\{1,\dots,2g\big\}$, the same computation implies $\Theta_{\rho}$ equals $\Theta'_\rho\cdot B$.\\
 \noindent It worth noticing $\Theta_\rho=\Theta_\rho\cdot AB$ and the matrices $A,B$ satisfy the equation $AB=\text{I}_{2g}$ implying that $A,B$ are unimodular. 
As the matrix $\Theta_\rho$ can be singular we cannot directly deduce that $AB=\text{I}_{2g}$, hence let us give a glimpse of why this is true. 

\smallskip

\noindent Instead of working in $\pi_1(S)$ we look at the situation in the first homology group $\mathrm{H}_1(S,\Z)\cong \Z^{2g}$. Let us consider $\alpha_1$ as a word $w\big(\alpha'_1,\beta'_1,\dots,\alpha'_g,\beta'_g\big)$; then
\[ \big[\alpha_1\big]=\big[\alpha'_1\big]^{b_{1\,1}}\big[\beta'_1\big]^{b_{1\,2}}\cdots\big[\alpha'_g\big]^{b_{1\,2g-1}}\big[\beta'_g\big]^{b_{1\,2g}}
\] where $b_{1j}$ with $j=1,\dots,2g$ are as above. On the other hand, any $\big[\alpha'_l\big]$ and $\big[\beta'_l\big]$ is of the form
\[ \big[\alpha'_l\big]=\big[\alpha_1\big]^{a_{2l-1\,1}}\big[\beta_1\big]^{a_{2l-1\,2}}\cdots\big[\alpha_g\big]^{a_{2l-1\,2g-1}}\big[\beta_g\big]^{a_{2l-1\,2g}},
\]
\[ \big[\beta'_l\big]=\big[\alpha_1\big]^{a_{2l\,1}}\big[\beta_1\big]^{a_{2l\,2}}\cdots\big[\alpha_g\big]^{a_{2l\,2g-1}}\big[\beta_g\big]^{a_{2l\,2g}}.
\] where $a_{ij}$ are as above. Replacing each $\big[\alpha'_l\big]$ and $\big[\beta'_l\big]$ inside $\big[w\big]=\big[\alpha_1\big]$, for any $l=1,\dots,g$, we obtain
\[ \big[\alpha_1\big]=\big[\alpha_1\big]^{k_1}\big[\beta_1\big]^{k_2}\cdots\big[\alpha_g\big]^{k_{2g-1}}\big[\beta_g\big]^{k_{2g}}.
\] As $\Z^{2g}$ is torsion-free, we may deduce that $k_1=1$ and $k_2=\cdots =k_{2g}=0$. On the other hand, it is straightforward to see that $k_m=\sum_{r=1}^{2g} b_{1r}a_{rm}$. Applying the same reasoning to any other generator we get the desire conclusion.

\begin{rmk}\label{mgm}
The matrices $A$ and $B$ found above may not have any geometrical meaning. Indeed, for closed surfaces the action of $\aut$ is not transitive on the set of basis of $\pi_1(S)$ and then two different basis may not be related by any automorphisms of $\pi_1(S)$. This means that not all matrices in $\slzz$ have a geometrical interpretation. As we shall see, a matrix has a geometrical meaning, that is induced by a homeomorphism of $S$, if and only if it is symplectic - see Proposition \ref{symacts} above.
\end{rmk}

\subsubsection{The effect of changing the basis of the lattice.}\label{ecl} We begin noticing that the $j$-th column of the matrix $\Theta_\rho$ corresponds to the vector of coordinates of a lift of the $j$-th generator of $\rho\big(\pi_1(S)\big)$ with respect to the standard lattice. Given any lattice $\Lambda$ with basis $\big\{v_1,\dots,v_n\big\}$ there is a matrix $g\in \text{SL}(n,\Z)$ such that $\Lambda=g\cdot\big(2\pi\Z^n\big)$. In particular $g\big(e_i\big)=v_i$. Change the basis means to change the coordinates of the vectors forming the columns of the matrix $\Theta_\rho$. Therefore, with respect to the lattice $\Lambda$, the matrix associated to $\rho$ has the following form $g \Theta_\rho$. In the sequel we shall need to consider the matrix $\Theta_\rho$ with respect to a lattice $\Lambda$ different to the standard one. We therefore extend the notation in the following way: We denote by $\Theta_\rho(\Lambda)$ the matrix associated to with respect to $\Lambda$. We shall use again the notation $\Theta_\rho$ when the lattice is the standard one.

\subsubsection{The $\Z$-row rank of the associated matrix.}  We now introduce the following numerical invariant concerning the associated matrix $\Theta_\rho$. As we shall see, such an invariant give us a way to characterise dense representations in $\Bbb T^n$ completely.

\begin{defn}\label{rora}
Let $M\in\textnormal{M}\big(n,m;\R\big)$. We define the $\Z$-row rank of $M$ as the dimension of the $\Z$-module generated by the rows of $M$. We shall denote it as $\text{rk}_{\Z}\big(M\big)$.
\end{defn}

\noindent We observe that the $\Z$-row rank is not invariant by transposition.

\begin{lem}\label{invrora}
Let $M\in\textnormal{M}\big(n,m;\R\big)$. The $\Z$-row rank $\textnormal{rk}_{\Z}\big(M\big)$ of $M$ is invariant under the left-action of $\textnormal{SL}(n,\Z)$. Similarly, $\textnormal{rk}_{\Z}\big(M\big)$ is invariant under the right-action of $\textnormal{SL}(m,\Z)$.
\end{lem}

\begin{proof}
We prove the first claim. Let $k=\text{rk}_{\Z}\big(M\big)\le n$. Define $Z$ as the subset of $\Z^n$ of those vectors $v$ such that $v M=0$. Notice that $Z$ is a $\Z$-module of dimension $n-k$. Let $A$ be any matrix in $\textnormal{SL}(n,\Z)$ and compute $AM$. It is easy to check that the $j$-th row is given by the linear combination $\sum_{i=1}^n a_{ji}\big(m_{i1},\dots,m_{im}\big)$. Suppose there is a vector $\mu=\big(\mu_1,\dots,\mu_n\big)$ such that $\mu A M=0$, then a straightforward computation shows that $\mu A\in Z$, that is $\mu=v A^{-1}$ for some $v\in Z$. Therefore, the subset $Z\cdot A^{-1}\subset \Z^n$ is the set of vectors $\mu$ such that $\mu A M=0$ and it has dimension $n-k$ over $\Z$. Therefore $\text{rk}_{\Z}\big(AM\big)=k$. Similarly, the second claim follows applying an analogous reasoning.
\end{proof}

\noindent Given a representation $\rho:\pi_1(S)\longrightarrow \T^n$, the following claims are direct consequences of the lemma above applied to the matrix $\Theta_\rho$.

\begin{cor}
Let $\rho:\pi_1(S)\longrightarrow \T^n$ be a representation and let $\Theta_\rho$ be the matrix associated to $\rho$ with respect to some basis of $\pi_1(S)$. The $\Z$-row rank of $\Theta_\rho$ is well-defined and it does not depend on any choice of a basis for $\pi_1(S)$ either on the choice of any lattice.
\end{cor}

\noindent Let $v_1\dots,v_k$ be vectors in $\R^n$. In the sequel, we shall make use of the following

\begin{defn}
We say that a $\Z$-module generated by $v_1\dots,v_k$ is $\pi\Q$-free if and only if 
\begin{equation}
    \langle v_1,\dots,v_k\rangle_{\Z}\cap \pi\,\Q^n=\big\{(0,\dots,0)\big\}.
\end{equation} 
\end{defn}

\noindent Keeping this definition in mind we finally state the following:

\begin{prop}\label{qfree}
Let $\rho:\pi_1(S)\longrightarrow \T^n$ be a representation and let $\Theta_\rho$ be the matrix associated to $\rho$ with respect to some basis of $\pi_1(S)$. The $\Z$-module $\langle\Theta_j\, :\,j=1,\dots,n\rangle_{\Z}$ generated by the rows of $\Theta_\rho$ is $\pi\Q$-free if and only if the $\Z$-module $\langle\big(A\Theta_\rho\big)_j\, :\,j=1\dots,n\rangle_{\Z}$ generated by the rows of $A\Theta_\rho$ is $\pi\Q$-free, where $A\in\textnormal{SL}(n,\Z)$. Similarly, $\langle\Theta_j\, :\,j=1,\dots,n\rangle_{\Z}$ is $\pi\Q$-free if and only if the $\Z$-module $\langle\big(\Theta_\rho B\big)_j\, :\,j=1,\dots,n\rangle_{\Z}$ is $\pi\Q$-free, where $B\in\textnormal{SL}(2g,\Z)$.
\end{prop}


\begin{proof}[Proof of Proposition \ref{qfree}]
Look at the matrix $A\Theta_\rho$ and suppose there are $\lambda_1,\dots,\lambda_n\in \Z$ such that 
\[ \sum_{j=1}^n\lambda_j\Big(\sum_{i=1}^n a_{ji}\big(\theta_{i,1},\dots,\theta_{i,2g}\big)\Big)\in \pi\Q^{2g}.
\] A simple manipulation of the formula above shows that
\[ \sum_{j=1}^n\lambda_j\Big(\sum_{i=1}^n a_{ji}\big(\theta_{i,1},\dots,\theta_{i,2g}\big)\Big)=\sum_{i=1}^n\Big(\sum_{j=1}^n \lambda_ja_{ji}\Big)\big(\theta_{i,1},\dots,\theta_{i,2g}\big),
\] implying the existence of some $\mu_1,\dots,\mu_n\in\Z$ such that $\displaystyle\sum_{j=1}^n\mu_j\big(\theta_{i,1},\dots,\theta_{i,2g}\big)\in\pi\Q^{2g}$. The proof of the second claim works similarly: Suppose there are $\lambda_1,\dots,\lambda_n\in \Z$ such that
\[ \sum_{i=1}^n\lambda_i\Big(\sum_{j=1}^{2g} \theta_{i,j}\big(b_{j1},\dots,b_{j2g}\big)\Big)\in \pi\Q^{2g}.
\] The same manipulation shows that
\[ \sum_{i=1}^n\lambda_i\Big(\sum_{j=1}^{2g} \theta_{i,j}\big(b_{j1},\dots,b_{j2g}\big)\Big)=\sum_{j=1}^{2g}\Big(\sum_{i=1}^{n} \lambda_i\theta_{i,j}\Big)\big(b_{i1},\dots,b_{i2g}\big),
\] implying that $\displaystyle\sum_{i=1}^{n} \lambda_i\theta_{i,j}\in\pi\Q$ for any $j=1,\dots,2g$. That is $\big(\lambda_1,\dots,\lambda_n\big)\Theta_\rho\in \pi\Q^{2g}$.
\end{proof}

\subsection{Remarks and comments on the modular action} In this section we collect a couple of final remarks about the $\spz$-action. 

\subsubsection{Explicit description of the modular action}\label{transact} The action of the mapping class group on the representation space is defined by pre-composition of any representation with an automorphism $\phi\in\textnormal{Out}(\pi_1(S))$, namely any representation $\rho$ is sent to $\rho\circ\phi^{-1}$ under this action. Since the Torelli group acts trivially on the representation space, the action of mapping class group boils down to an action of the group $\spz$ which agrees with the $\spz$-action on \textnormal{Hom}$\big(\mathrm{H}_1(S,\Bbb Z),\Bbb T^n)$, this is a consequence of Proposition \ref{torga}. In section \S\ref{matrep}, we have identified the representation space with $\text{M}\big(n,2g;\T\big)$ by using the mapping $\imath$ associating to any representation $\rho$ its matrix $\Theta_\rho$. We use such a mapping to transfer the action of Mod$(S)$ on \textnormal{Hom}$\big(\pi_1(S),\T^n\big)$ to an action of $\spz$ on $\text{M}\big(n,2g;\T\big)$. Since any $\phi\in\mathcal{I}(S)$ leaves $\rho$ fixed, the matrix associated to $\rho'=\phi\cdot\rho=\rho\circ\phi^{-1}$ agrees with $\Theta_\rho$, this is a consequence of Proposition \ref{torga}. Any coset $\phi\,\mathcal{I}(S)$ defines a unique matrix $A$ in $\spz$. In the light of the discussion given at subsection \S\ref{ecb}, the matrix associated to $\rho'=\phi\cdot\rho$ is $\Theta_{\rho'}=\Theta_\rho\,A^{-1}$. Therefore, the action of $\spz$ on $\text{M}\big(n,2g;\T\big)$ is defined as $A\cdot\Theta_\rho=\Theta_\rho\,A^{-1}$. As the mapping $\imath$ is a homeomorphism, it is clear that $\spz$-orbit of $\rho$ is dense in the representation space if and only if the Mod$(S)$-orbit of $\Theta_\rho$ is dense in $\text{M}\big(n,2g;\T\big)$.

\subsubsection{The modular action commutes with the change of lattice} Given a representation $\rho:\pi_1(S)\to\T^n$, the main goal of the present paper is to study its orbit under the action of the mapping class group. This reduces to study the orbit of the matrix $\Theta_\rho$ naturally attached to $\rho$ in the space M$\big(n,2g;\Bbb T\big)$. However the matrix $\Theta_\rho$ depends on the lattice chosen and hence the orbit could depend on the chosen lattice. The aim of this paragraph is to point out that this is not the case; indeed the change of lattice commutes with the modular action.  Each element $\Theta$ in the space M$\big(n,2g;\Bbb T\big)$ can be thought as the datum of $n$ vectors $\Theta_i\in\T^{2g}$ corresponding to the rows of $\Theta$. By adopting this point of view, the space M$\big(n,2g;\Bbb R\big)$ identifies with $\T^{2g}\times\cdots\times\T^{2g}$. There is a left action of the group $G$ defined as \\
\[ G=\left\{\begin{pmatrix} A &  & \\ & \ddots &\\ & & A \end{pmatrix} \, : \, A\in\spz\right\} \cong\spz <\text{SL}(2gn,\Z)
\] 
\text{}\\
\noindent on the $2gn$-dimensional torus induced by the natural right action of the symplectic group in the matrix space M$\big(n,2g;\Bbb T\big)$. Using this new perspective, one can easily verify that any change of lattice commutes with the $\spz$ action. Indeed, any change of lattice $h\in\slzn$ can be seen as an element of the group $H$ defined as
\[ H=\left\{\begin{pmatrix} 
h_{1\,1}\text{I}_{2g} & \cdots & h_{1\,n}\text{I}_{2g}\\
\vdots & \ddots & \vdots\\
h_{n\,1}\text{I}_{2g} & \cdots & h_{n\,n}\text{I}_{2g}
\end{pmatrix}\, \Bigg| \, \text{ where } \begin{pmatrix} h_{1\,1} & \cdots  & h_{1\,n}  \\ \vdots & & \vdots \\ h_{n\,1} & \cdots & h_{n\,n} \end{pmatrix}\in\slzn\right\} \cong\slzn <\text{SL}(2gn,\Z)
\] Since $H$ commutes with the group $G$ defined above, the $\spz$ action commutes with the change of lattice.

\medskip

\section{Characterising dense representations}\label{dr}

\noindent In this section we provide a complete characterisation of \emph{dense representations} by providing necessary and sufficient conditions. From section \S\ref{intro}, we recall that a representation $\rho:\pi_1(S)\longrightarrow \T^n$ is \emph{dense} if the subgroup $\rho\big(\pi_1(S)\big)$ is dense in $\T^n$. We have seen in the previous section that, upon choosing a basis of the fundamental group and a lattice, each representation is represented by a well-define matrix $\Theta_\rho$. Along this section we fix an arbitrary basis for the fundamental group and we consider $\T^n$ as the quotient of $\R^n$ with the standard lattice. 

\begin{thm}\label{codr}
Let $\rho:\pi_1(S)\longrightarrow \Bbb T ^n$ be a representation and let $\Theta_\rho$ its associated matrix. Then $\rho$ is dense in $\Bbb T^n$ if and only if \textnormal{rk}$_{\Z}\big(\Theta_\rho\big)=n$ and the rows of $\Theta_\rho$ generate a $\pi\mathbb{Q}$-free $\mathbb{Z}$-module.
\end{thm}

\noindent NWe can notice that the necessary condition means that the $\Z$-module generated by the rows of the matrix $\Theta_\rho$ does not intersect $\pi\Q^{2g}$ and is equivalent to say that row rank over $\Z$ of the matrix
\begin{equation}\label{commat}
\begin{pmatrix}
\Theta_\rho \\ 
\pi\cdot\text{I}_{2g}
\end{pmatrix}=
\begin{pmatrix}
\theta_{1,1} & \cdots & \theta_{1,2g}\\
\vdots &  & \vdots\\
\theta_{n,1} & \dots & \theta_{n,2g}\\
\pi & \cdots & 0 \\
\vdots & \ddots & \vdots \\
0 &\cdots& \pi
\end{pmatrix}
\end{equation}
\noindent is maximal, namely $2g+n$. Before proving the Theorem, we need a preliminar Lemma.

\begin{lem}\label{denproj}
Suppose that $\rho:\pi_1(S)\longrightarrow \Bbb T^n$ is dense. Then each representation $\rho_k=\pi_k\circ\rho$, where $\pi_k$ is the projection to $k^{th}$ factor, is dense.
\end{lem}

\begin{proof}[Proof of Lemma \ref{denproj}]
Suppose there is $k$ for which the representation $\rho_k$ is not dense. Then there is an open subset $A\subset \Bbb S^1$ such that $A\cap\rho_k\big(\pi_1(S)\big)=\emptyset$. Suppose without loss of generality that $k=1$. Then $\big(A\times \Bbb T^{n-1}\big)\cap \rho\big(\pi_1(S)\big)=\emptyset$. In particular $\rho$ is not dense, hence a contradiction.
\end{proof}

\begin{proof}[Proof of Theorem \ref{codr}]
Assume $\rho$ has a dense image and suppose the $\Z$-module generated by the rows intersect $\pi\Q^{2g}$ that is $\text{rk}_{\Z} \begin{pmatrix}
\Theta_\rho \\ 
\pi\cdot \text{I}_{2g}
\end{pmatrix} < 2g +n$. Thus, there is a row $\Theta_i$ of $\Theta$ such that: 

\[\sum_{\substack{j=1 \\j\neq i}}^{n}\lambda_j \Theta_j + \big(\lambda_{n+1}\pi,\dots,\lambda_{n+2g}\pi\big)=\lambda_i \Theta_i
\]

\noindent with $\lambda_i$ different to zero. Such a summation can be rewritten as: 

\[\sum_{j=1}^{n}\lambda_j \Theta_j = \big(\lambda_{n+1}\pi,\dots,\lambda_{n+2g}\pi\big)
\]

\noindent for some $\lambda_j \in \mathbb{Z}$ and not all zero. Consider the matrix $M\in \text{M}\big(n,\mathbb{Z}\big)\cap \text{GL}\big(n,\Q\big)$ defined as:\\
\[ M=I_n+ (\lambda_i-1)E_{ii} - \sum_{\substack{j=1 \\j\neq i}}^n \lambda_j E_{ij}=
\begin{pmatrix}
1 & 0 & \cdots & \cdots & \cdots  & 0\\
0 & 1 & \cdots & \cdots & \cdots & \cdots \\
\cdots & \cdots & \ddots & \cdots & \cdots & \cdots\\
-\lambda_1 & -\lambda_2 & \cdots & \lambda_i & \cdots & -\lambda_n\\
\cdots & \cdots & \cdots & \cdots  & 1 & 0\\
0 & 0 & \cdots & \cdots & 0 & 1\\
\end{pmatrix}
\]
\text{}\\
\noindent  where the $E_{ij}=\big(e_{kl}\big)$ are the matrices with coefficients $e_{kl}= \delta_{ki}\delta_{lj}$. The matrix $M$ defines a linear homeomorphism, say $f_M$ of $\R^n$ with respect to the canonical basis because $\det M=-\lambda_i$ which is different to zero. The mapping $f_M$ sends $\Z^n$ to itself and descends to a finite-degree covering $\overline{f}_M:\T^n\longrightarrow \T^n$ - in fact the degree coincides with the determinant of $M$. In particular the following equation holds: $\pi\circ f_M=\overline{f}_M\circ \mathrm{exp}$, where $\mathrm{exp}:\R^n\longrightarrow \T^n$ denotes as usual the canonical projection. Consider now the $\Z$-module $\langle\Theta_j\, :\,j=1,...,n\rangle_{\Z}$ generated by the rows of $\Theta_\rho$. A straightforward computation shows that its image via the mapping $f_M$ is the $\Z$-module generated by the vectors
\[\Bigg\langle
\begin{pmatrix}
\theta_{1,1}\\
\dots\\
\lambda_{n+1}\pi\\
\dots\\
\theta_{n,1}
\end{pmatrix}
,\dots,
 \begin{pmatrix}
\theta_{1,2g}\\
\dots\\
\lambda_{n+2g}\pi\\
\dots\\
\theta_{n,2g}
\end{pmatrix}
\Bigg\rangle_{\Z}\]

\noindent Let us point out the following fact: As $\rho$ is assumed to be dense in the torus, the image via the canonical projection in $\mathbb{T}^n$ of $\mathbb{Z}$-module generated by the rows of $\Theta_\rho$ fills a dense subset of $\mathbb{T}^n$, namely the image of $\rho$. As $M$ commutes with the action of $2\pi\Z^n$ and pass through to the quotient as a finite-degree covering map of the $\mathbb{T}^n$, the $\Z$-module $M\cdot\langle\Theta_j\, :\,j=1,...,n\rangle_{\Z}$ is mapped on a dense subset of the torus. On the other hand, the projection of the $i$-th factor is discrete. Lemma \ref{denproj} implies the desire contradiction.\\
We now prove the opposite implication and again we argue by contradiction. Suppose $\rho$ does not have a dense image in the $n$-torus, then its closure is a $k$-dimensional sub-manifold, say $S_0$, of dimension $k<n$. We note that $S_0$ may not be connected in general. Indeed any closed subgroup of $\T^n$ is homeomorphic to $\T^d\times\frac{\Z}{ m_1\Z}\times\cdots\times\frac{\Z}{m_{n-d}\Z}$, that is a finite collection of inhomogeneous torii. Assume first $S_0$ be connected; we shall deduce the general case later on. The subspace $S_0$ lifts to a linear subspace $\widetilde{S_0}$ of $\R^n$ which of course contains the $\Z$-module $\langle\Theta^j\, :\,j=1,...,2g\rangle_{\Z}$ generated by the columns of $\Theta_\rho$. We now invoke the following lemma.

\begin{lem}\label{ccm}
There is $g\in \slzn$ such that : 
\[g\cdot \langle\Theta^j\, :\,j=1,\dots,2g\rangle_{\Z} < \big\langle e_1,\dots,e_k\big\rangle_\R\]
where the $e_i$'s are the vectors of the canonical basis of $\R^n$.
\end{lem}

\noindent Assume the lemma holds. The $\Z$-module $g\cdot \langle\Theta^j\, :\,j=1,...,2g\rangle_{\Z}$ is contained in the first factor of $\mathbb{T}^n=\mathbb{T}^k \times\mathbb{T}^{n-k}$ and then $\Theta_\rho$ cannot have maximal row rank over $\Z$. As a consequence the matrix given in the equation \ref{commat} cannot have maximal row rank over $\Z$. The general case follows by applying the same reasoning to the component $S_0^o$ of $S_0$ containing the identity which contains a finite-index $\Z$-module of $\langle\Theta^j\, :\,j=1,...,2g\rangle_{\Z}$. In the general case, $g\cdot \langle\Theta^j\, :\,j=1,...,2g\rangle_{\Z}$ is contained in $\mathbb{T}^n=\mathbb{T}^k \times F$, where $F$ is isomorphic to the finite group $\frac{\Z}{ m_1\Z}\times\cdots\times\frac{\Z}{m_{n-d}\Z}$. Let us proceed with the proof of Lemma \ref{ccm}.\\

\noindent \emph{Proof of Lemma \ref{ccm}.} If $\widetilde{S_0}^o$ is contained in $\big\langle e_{\sigma(1)},\dots,e_{\sigma(k)}\big\rangle_\R$ for some $\sigma\in \mathfrak{S}_n$ then it is sufficient to rename the coordinates. This corresponds to a matrix $g$ obtained by product of elementary matrices. Assume $\widetilde{S_0}^o$ is not contained in any such a space. Let $x_i$ be the intersection of $\widetilde{S_0}^o$ with the affine space $e_i + \mathbb{R}^{n-k}$ and let $d_i$ its Euclidean distance to $\mathbb{R}^d$. Then $x_i$ has the following form:
\[x_i=\big(0,\dots,1,\dots,0,t_1,\dots,t_{n-k}\big)\]
where $t_i\in\Q$. In fact, if this had been not true then $S_0^o$ would have been a dense  subspace of dimension $k+1$ in the torus. As a consequence $d_i\in \Q$ for any $i=1,\dots,k$ and $\widetilde{S_0}^o$ is described by $n-k$ equations with integer coefficients.
Look at the set $\widetilde{S_0}^o\,\cap\,\Z^n$. This is a lattice in $\widetilde{S_0}^o$ and there is a basis $v_1,\dots,v_k$ made of integer vectors. 
We invoke \cite[Corollary 3, pag.14]{CJ} to claim the existence of $n-k$ vectors $v_{k+1},\dots,v_n$ such that the vectors $v_1,\dots,v_n$ gathered together form a basis for $\Z^n$. Since $\slzn$ acts transitively on the space of lattices, there is $g$ such that 
\[g\cdot \langle\Theta^j\, :\,j=1,\dots,2g\rangle_{\Z} <g\cdot \widetilde{S_0} =\big\langle e_1,\dots,e_k\big\rangle_\R.
\]  This concludes the proof of Lemma \ref{ccm} and indeed the proof of Theorem \ref{codr}.
\end{proof}

\noindent From the proof we deduce that the row rank of the matrix $\Theta_\rho$ has a very explicit geometric interpretation, in fact it coincides with the dimension of the subspace containing the image of $\rho$. Of course, the proof does not depend on the presentation of $\pi_1(S)$ either on the lattice chosen. Let us prove these facts.

\begin{proof}[Independence on the chosen basis] Let $\langle\Theta^j\, :\,j=1,\dots,2g\rangle_{\Z}$ be the $\Z$-module generated by the columns of $\Theta_\rho$. In section \S\ref{ecb} we have seen that the effect of changing a basis of generators corresponds to multiply on the right the matrix $\Theta_\rho$ with a matrix $A\in\text{SL}(2g,\Z)$. Since the row rank of $\Theta_\rho$ is invariant under the action by right-multiplication of $\text{SL}(2g,\Z)$, the matrices $\Theta_\rho$ and $\Theta_\rho A$ have the same row rank. Furthermore, in the light of Corollary \ref{qfree}, $\pi\Q$-freedom is also invariant under the right action of $\slzz$. On the other hand, let $S$ be the closure of the subspace of $\T^n$ generated by the columns of $\Theta_\rho$. Its lift $\widetilde{S}$ is a linear (possibly improper) subspace of $\R^n$ described by $n-k$ equations. As the columns of $\Theta_\rho A$ satisfy the same equations, the image of unaffected by the change of basis. This proves the independence on the basis chosen.
\end{proof}

\begin{proof}[Independence on the lattice chosen] Given two lattices $\Lambda_1$ and $\Lambda_2$, there always exists an element of $A\in\text{SL}(n,\Z)$ mapping the first lattice on the second one because the action of $\text{SL}(n,\Z)$ is transitive on the set of lattices. Such a map descends to a homeomorphism of the $n$-torus and hence the $\Z$-module $\langle\Theta^j\, :\,j=1,\dots,2g\rangle_{\Z}$ projects to a dense subset of the torus if and only if its image via $A$ projects to dense subset as well. On the other hand, it is immediate to verify that the row rank of the matrix $\Theta_\rho(\Lambda_1)$ equals the one of $\Theta_\rho(\Lambda_2)$ because the row rank is invariant under the action by left-multiplication of $\text{SL}(n,\Z)$. Again, Corollary \ref{qfree} implies $\pi\Q$-freedom is invariant under the left action of $\slzn$. Hence the conclusion.
\end{proof}

\noindent We finally provide a couple of explicit examples.

\begin{ex}\label{ex1}
Let $S$ be a surface of genus $2$, and let $\rho:\pi_1(S)\longrightarrow \Bbb T^2\cong\Bbb S^1\times\Bbb S^1$ be the representation such that $\rho(a_1)=\rho(a_2)=\big( e^{i\varphi}, e^{i\varphi} \big)$, where $\varphi\in\Bbb R\setminus \pi\Bbb Q$, and $\rho(b_1)=\rho(b_2)=(1,1)$.\\

\noindent The matrix $\Theta$ has the following form
\[
\begin{pmatrix}
\varphi & 0 & \varphi & 0\\
\varphi & 0 & \varphi & 0
\end{pmatrix}
\] If $\gamma\in\pi_1(S)$, then $\rho(\gamma)=\rho(a_1)^{k_1}\rho(b_1)^{k_2}\rho(a_2)^{k_3}\rho(b_2)^{k_{4}}$ with $k_i\in\Bbb Z$. Consider the vector $v=(k_1,k_2,k_3,k_4)$, then 
\begin{equation}
    \Theta \cdot v=\big((k_1+k_3)\varphi,(k_1+k_3)\varphi\big).
\end{equation}

\noindent Viewing the $2$-torus as a complex with one $0$-cell, four $1$-cells and one $2$-cell, the image of $\rho$ is densely contained the main diagonal. Both projections are dense in $\Bbb S^1$, but the image does not fill $\Bbb T^2$. Notice that the row rank of $\Theta$ over $\Z$ is one as the dimension of the smallest subspace containing $\rho\big(\pi_1(S)\big)$. 
\end{ex}

\begin{ex}\label{ex2}
Let $S$ be a surface of genus $2$, and let $\rho:\pi_1\Sigma\longrightarrow \Bbb T^2\cong\Bbb S^1\times\Bbb S^1$ be the representation such that $\rho(a_1)=\rho(a_2)=\big(e^{i\varphi},1\big)$ and $\rho(b_1)=\rho(b_2)=\big(1,e^{i\varphi}\big)$ with $\varphi\in\Bbb R\setminus \pi\Bbb Q$.\\

\noindent The matrix $\Theta$ has the following form
\[
\begin{pmatrix}
\varphi & 0 & \varphi & 0\\
0 & \varphi & 0 & \varphi
\end{pmatrix}
\] If $\gamma\in\pi_1(S)$, then $\rho(\gamma)=\rho(a_1)^{k_1}\rho(b_1)^{k_2}\rho(a_2)^{k_3}\rho(b_2)^{k_{4}}$ with $k_i\in\Bbb Z$. Consider the vector $v=(k_1,k_2,k_3,k_4)$, then 
\begin{equation}
    \Theta \cdot v=\big((k_1+k_3)\varphi,(k_2+k_4)\varphi\big).
\end{equation}

\noindent Viewing the $2$-torus as a complex with one $0$-cell, four $1$-cells and one $2$-cell, the image of $\rho$ densely fills the torus. Notice that the rank of $\Theta$ is two in this case and both projections are dense.
\end{ex}

\section{$\spz$-action and orbit closures}\label{ratthm}

\noindent The symplectic group $\spz$ acts on the homological representation space \textnormal{Hom}$\big(\mathrm{H}_1(S,\Z),\T^n\big)$ by precomposition. We have seen in section \ref{matrep} that, up to a choice of a symplectic basis, this latter space identifies with the space $\text{M}\big(n,2g;\T\big)$. In this section we would like to study the orbit closures of an element of $\text{M}\big(n,2g;\T\big)$ under the action of $\spz$. The first thing we notice is that a subset $\Omega\subset \text{M}\big(n,2g;\R\big)$ is invariant under the action $\spz\ltimes \text{M}\big(n,2g; 2\pi\Z\big)$ if and only if its projection onto $\text{M}\big(n,2g;\T\big)$ is $\spz$-invariant. This simple remark legitimates us to study the orbit closures on the universal cover, that is $\text{M}\big(n,2g;\R\big)$.

\smallskip

\noindent Let us consider the group $G=\spr\ltimes \text{M}\big(n,2g;\R\big)$. Given two elements $(A,\,a)$ and $(B,\,b)$, their product is defined as follows $ (A,\,a)\cdot(B,\,b)=(AB,\, bA^{-1}+a)$. The group $G$ acts transitively on the space $\text{M}\big(n,2g;\R\big)$ with the action being defined as $(A,\,a)\cdot p= pA^{-1}+a$ $-$ indeed a point $p\in \text{M}\big(n,2g;\R\big)$ may be regarded as the couple $(I,p)$. There is a natural identification between the space $\text{M}\big(n,2g;\R\big)$ with the $G/U$, where $U$ is the stabiliser of any point. It is straightforward to check that the of the zero matrix is nothing but $\spr$. The subgroup $\Gamma=\spz\ltimes \text{M}\big(n,2g; 2\pi\Z\big)$ is a lattice in $G$ and acts in the obvious way on $G/U$. Under these conditions we are in the right position to apply Ratner's Theorem, see \cite{RM}, which we state as follows according to our setting.

\begin{rthm}
Let $G,U,\Gamma$ as above and let $p\in \text{M}\big(n,2g;\R\big)=G/U$ such that $p=\gamma\,U$. Then there is a closed subgroup $H_\gamma$ such that the following holds.
\begin{itemize}
\item $U_\gamma=\gamma\,U\,\gamma^{-1} \le H_\gamma$,
\item $\Gamma\,\cap\,H_\gamma$ is a lattice in $H_\gamma$ and
\item $\overline{\Gamma\cdot p}=\Gamma\,H_\gamma\,p$.
\end{itemize}
\end{rthm}

\noindent Notice that $\gamma$ can be taken as $(I,\,p)$. Since our goal here is to classify the closures of $\Gamma$-orbits of any point in the space of matrices $\text{M}\big(n,2g;\R\big)$, we just need to figure out which subgroups of $G$ may be provided by Ratner's Theorem. To this purpose, let us consider the projection $\Phi:\spr\ltimes \text{M}\big(n,2g;\R\big) \longrightarrow \spr$. Given a point $p$ in $\text{M}\big(n,2g;\R\big)$,  the group $H_\gamma$ is isomorphic to the semidirect product $H_\gamma=U_\gamma\ltimes K_\gamma$, where $K_\gamma$ is defined as $\textnormal{ker}\,\Phi\,\cap H_\gamma$, that is the kernel of the mapping $\Phi$ restricted to $H_\gamma$. Notice that the image of $H_\gamma$ under the mapping $\Phi$ is the whole group $\spr$ because $H_\gamma\ge U_\gamma\cong\spr$. In particular, $H_\gamma\cong \spr\ltimes K_\gamma$. \\

\noindent Let us proceed on understanding $K_\gamma$. The first thing we notice is that any change of lattice $h\in\slzn$ extends to a homeomorphism $\phi_h$ of $G$ defined as
\[ \phi_h:G\longrightarrow G, \qquad \phi_h(A,\,a)=(A,ha).
\] 
\noindent This is an automorphism of $G$ and its restriction to \textnormal{ker}$\,\Phi$, where $\Phi$ is the projection just defined above, is linear and corresponds to a change of lattice in $\text{M}\big(n,2g;\R\big)$. In particular, the relation
\begin{equation}
\Gamma\cdot(h\,p)=\phi_h\big(\Gamma\cdot p\big)
\end{equation} holds for any $p\in\text{\emph{M}}\big(n,2g;\R\big)$. As a consequence of Lemma \ref{ccm}, there is an element $h\in\slzn$ such that $h\cdot p$ is of the following form 
\begin{equation}\label{redmat}
\begin{pmatrix}
\theta_{1,1} & \cdots & \theta_{1,2g}\\
\vdots &  & \vdots\\
\theta_{k,1} & \dots & \theta_{k,2g}\\
\pi\, q_{k+1,1} & \cdots & \pi\, q_{k+1,2g} \\
\vdots &  & \vdots \\
\pi\, q_{n,1} &\cdots& \pi\, q_{n,2g}
\end{pmatrix}=
\begin{pmatrix}
\Theta_o\\
\pi Q
\end{pmatrix}
\end{equation}
\noindent  where 
\begin{itemize}
\item $\Theta_o\in \text{M}\big(k,2g;\R\big)$ for some $0\le k\le n$,
\item $\pi Q\in \text{M}\big(n-k,2g;\pi\Q\big)$,
\item the vectors $\big\{(\theta_{i,1},\dots,\theta_{i,2g})\big\}_{i=1,\dots,k}$ lines are linearly independent over $\Bbb Z$, and
\item $\big\langle(\theta_{i,1},\dots,\theta_{i,2g})\,:\,i=1,\dots,k\big\rangle$ is $\pi\Q$-free;
\end{itemize}
\noindent and hence it is sufficient to study $K_\gamma$ for $\gamma=(I,p)$ and $p$ is a matrix in the form \ref{redmat}. Furthermore, it will be sufficient to study the closures of $\Gamma$-orbits for matrices in these form. We also notice that $K_\gamma$ is a linear subspace of $\text{M}\big(n,2g;\R\big)$ invariant under the action of $U_\gamma$ by conjugation. In fact, suppose $\gamma=(I,\,p)$, let $q=(I,\,q)\in K_\gamma$ be any point and let $(A, p-pA^{-1})$ be a generic element of $U_\gamma$. Then 
\[ (A, p-pA^{-1})\cdot(I,\,q)\cdot(A^{-1}, p-pA)=(I,\,qA^{-1})\in K_\gamma
\] as claimed. The following Lemma implies our main Theorem \ref{thmbf} for representations of closed surface groups into the unit circle $\Bbb S^1$.

\begin{lem}\label{Lemma 4.1}
Let $\overline{\theta}=\big(\theta_1,\dots,\theta_{2g}\big)\in\Bbb R^{2g}$. If $\overline{\theta}\in\pi\Bbb Q^{2g}$, then $\big(\spz\ltimes2\pi\Z^{2g}\big)\cdot\overline{\theta}$ is discrete in $\R^{2g}$. Otherwise, if $\overline{\theta}\in\R^{2g}\setminus \pi\Q^{2g}$, then $\spz\cdot \overline{\theta}$ is dense in $\T^{2g}$ and hence $\big(\spz\ltimes2\pi\Z^{2g}\big)\cdot\overline{\theta}$ is dense in $\R^{2g}$.
\end{lem}

\begin{proof}
Let $\Lambda$ be the subgroup of $\R$ generated by the entries of $\overline{\theta}$ and consider $\Lambda^{2g}$. The first claim is easy to establish. In this case $\Lambda^{2g}$ is a lattice in $\R^{2g}$ containing $2\pi\Z^{2g}$ and preserved by $\spz$. Now observe that the $\spz$-orbit of $\overline{\theta}$ is contained in $\Lambda^{2g}$. Suppose $\overline{\theta}\in\R^{2g}\setminus\pi\Q^{2g}$, hence there exists $\theta_i\in\R\setminus\pi\Q$. There is an element in $\spz$ such that all the entries are $\R\setminus\pi\Q$. We may assume $\theta_i\in\,[0,2\pi)$ for all $i=1,\dots,2g$. Let $\theta^*\in\Bbb T^{2g}$ be any point. For each couple $(\theta_{2i-1},\,\theta_{2i})$, where $i=1,\dots,g$, there are two integers $k_i,\,h_i$ such that the couple $\big(k_i\,\theta_{2i-1}+\theta_{2i},\,(k_i\,h_i-1)\theta_{2i-1}+h_i\,\theta_{2i}\big)$ is closed to $(\theta_{2i-1}^*,\,\theta_{2i}^*)$. Therefore the $\spz$-orbit of $\overline{\theta}$ is dense in $\T^{2g}$ and hence $\spz\ltimes2\pi\Z^{2g}\cdot\overline{\theta}$ is dense in $\R^{2g}$ as desired.
\end{proof}

\noindent Before proving the general case we need the following proposition on which we describe the group $K_\gamma$.

\begin{prop}\label{kgclass}
Let $p\in X$ be any point in the form given in the equation \eqref{redmat} and let $k$ the number of lines not in $\pi\Q^{2g}$. Let $H_\gamma$ be the group provided by Ratner's Theorem, where $\gamma=(I,\,p)$. Then $K_\gamma$ is trivial or $K_\gamma=\text{\emph{M}}\big(k,2g;\R\big)$.
\end{prop}

\begin{proof}[Proof of Proposition \ref{kgclass}]
Let $p$ be any point in $\text{M}\big(n,2g;\R\big)$. Assume $p$ be different from the zero matrix for which the claim trivially holds. Let us begin with the case $p=\pi Q\in\text{M}\big(n,2g;\pi\Q\big)$, that means $k=0$. We claim $K_\gamma$ to be trivial. Let $\gamma=(I,p)$ and let $H_\gamma$ be the group provided by Ratner's Theorem. The orbit $\Gamma\cdot p$ lies in the subgroup of $\text{M}\big(n,2g;\R\big)$ generated by the matrices $\pi q_{ij}\,E_{ij}$, where $\pi q_{ij}$ are the entries of $p$, which is discrete and closed. This means that $\overline{\Gamma\cdot p}=\Gamma\cdot p$ and implies $H_\gamma$ is the stabiliser of $p$. Therefore $H_\gamma=U_\gamma$ and hence $K_\gamma$ is trivial. Notice that this argument generalises the first case of the previous Lemma \ref{Lemma 4.1}. Let us now assume $k>0$. The linear space $K_\gamma$ is completely determined by $\Theta_o$, indeed the block $\pi Q$ does not give any contribution. In this case, the orbit $\Gamma\cdot p$ is no longer closed and the $\spz$-orbit of $p$ is contained in some linear subspace of $\text{M}\big(k,2g;\R\big)$ of dimension $2g\,l$, where $l$ is the dimension of the linear space generated by the rows of $\Theta_o$. Hence $K_\gamma$ contains $V$ as a proper subspace. We can notice that $V$ is $\spr$-invariant but $V\,\cap\,\text{M}\big(k,2g;2\pi\Z\big)$ is not a lattice because the $\Z$-module generated by the rows of $\Theta_o$ is $\pi\Q$-free. For the same reason, the minimal linear space containing $V$ and a lattice is $\text{M}\big(k,2g;\R\big)$, hence $K_\gamma=\text{M}\big(k,2g;\R\big)$.
\end{proof}

\noindent From the proof of Proposition \ref{kgclass} we can deduce the following corollary.

\begin{cor}
Let $p\in \text{\emph{M}}\big(n,2g;\R\big)$ be any point in the form given in the equation \eqref{redmat} and let $k$ the number of lines not in $\pi\Q^{2g}$. There exists a closed connected subgroup $H\le \Bbb T^n$ of dimension $k$ such that $\overline{\Gamma\cdot p}$ projects to a finite union of inhomogeneous torii of dimension $k$ corresponding to cosets of $H$. In particular, the modular orbit of a dense representation $\rho:\mathrm{H}_1(S,\Z)\longrightarrow \Bbb T^n$ is dense in the representation space.
\end{cor}

\noindent This corollary implies Theorem \ref{thmbf2} and indeed Theorem \ref{thmbf}. In the appendix, we shall study the modular orbits by applying a direct approch without rely on Ratner's Theory. 

\medskip

\section{An application: Approximation result}\label{apres}

\noindent The aim of this final section consists in showing Proposition \ref{aiffd} and indeed Theorem \ref{thmbf4}. Let us begin by recalling the statement of Kronecker's Theorem as formulated in \cite[Section 26.19(e)]{HR}. The reader may also consult \cite[Section 1.12(iii)]{BM} for another one-dimensional version of Kronecker's theorem.

\begin{kthm}\label{kthm}
Let $b^{(i)}=\big(b_1^{(i)},\dots,b_m^{(i)}\big)$, with $i=1,\dots,n$, be vectors of $\R^m$ such that $b^{(1)},\dots,b^{(n)},$ $\pi e_1,\dots,\pi e_m$ are linearly independent over $\Q$ in the vector space $\R^m$ (where the $e_j$'s form the canonical basis of $\R^m$). Let $a_1,\dots,a_n$ be any real numbers and let $\varepsilon$ be a positive number. Then there is an element $\big(k_1,\dots,k_m\big)\in\Z^m$ such that
\begin{equation}\label{keq}
 \Big| a_i-\sum_{l=1}^m k_lb_l^{(i)}\Big|<\varepsilon \quad \textnormal{mod } 2\pi
\end{equation} for every $i=1,\dots,n$.
\end{kthm}

\noindent For a real $a$, the expression $|a|<\varepsilon \textnormal{ mod } 2\pi$ means that $|a-2k\pi|<\varepsilon$ for some integer $k$. From the equation \eqref{keq} above, one can easily infer the equivalent estimate
\begin{equation}
\Big| \Big| \,\big(a_1,\dots,a_n\big)^t-2\pi\big(h_1,\dots,h_n\big)^t-B\,\big(k_1,\dots,k_m\big)^t\,\Big| \Big|<C\varepsilon
\end{equation} where $\big(h_1,\dots,h_n\big)\in\Z^n$, $B$ is the matrix having $b^{(i)}$'s as rows, $C$ is a real constant depending only on $n$ and $| |\cdot||$ is any norm on $\mathbb{R}^n$. Kronecker's theorem generalises to simultaneous approximation of $l$ given real vectors $a^{(j)}=\big(a_{1j},\dots,a_{nj}\big)^t$ where $j=1,\dots,l$. Indeed, for any $\varepsilon>0$ there is a matrix $K\in\text{M}\big(m,l;\Z\big)$ such that 
\begin{equation}
    \Big| \Big|\, A-2\pi H -B\, K\,\Big| \Big|<C\varepsilon
\end{equation} 
 where $A$ is the matrix having $a^{(j)}$'s as columns, $H\in\text{M}\big(n,l;\Z\big)$ and $C$ is a constant depending only on $l,n$. That is 
\begin{equation}
    \Big| \Big|\, A -B\, K\,\Big| \Big|<\varepsilon\textnormal{ mod } 2\pi. 
\end{equation}
\smallskip

\noindent Let $S$ be a closed surface of genus greater than zero, let $\rho:\pi_1(S)\longrightarrow \T^n$ be a representation and let $\Theta_\rho$ be the associated matrix in the sense of Definition \ref{mr}. 

\begin{propnn} The following are equivalent.
\SetLabelAlign{center}{\null\hfill\textbf{#1}\hfill\null}
\begin{itemize}[leftmargin=1.75em, labelwidth=1.5em, align=center, itemsep=\parskip]
\item[ \textnormal{1.}] \textnormal{Mod}$(S)\cdot \rho$ is dense in the representation space.
\item[ \textnormal{2.}]  For any matrix $A\in\textnormal{M}\big(n,2g;\R\big)$ and any $\varepsilon>0$ there is a matrix $g\in\spz$ such that 
\[\big| \big| A-\Theta_\rho\, g\big| \big|<\varepsilon\textnormal{ mod } 2\pi.
\]
\end{itemize} 
\end{propnn}

\begin{proof}[Proof of Proposition \ref{aiffd}]
Each representation is identified with its associated matrix and the representation space with $\text{M}\big(n,2g;\T \big)$. Suppose \textnormal{Mod}$(S)\cdot \Theta_\rho$ is not dense in the representation space. Then there is an open set $U$ such that \textnormal{Mod}$(S)\cdot \Theta_\rho\,\cap\, U=\phi$. Let $A$ be any matrix in $U$ and $\varepsilon$ a strictly positive real number such that the open ball $B_\varepsilon(A)\subset U$. Then, for any $g\in\textnormal{Mod}(S)$ the following estimate $\big| \big| A-\Theta_\rho\, g\big| \big|>\varepsilon\textnormal{ mod } 2\pi$ holds. As the action of the Torelli subgroup is trivial by Proposition \ref{torga}, the action of the mapping class group coincides with the action of $\spz$. Therefore, Theorem \ref{thmbf4} implies Theorem \ref{thmbf}.

\smallskip

\noindent Suppose \textnormal{Mod}$(S)\cdot \Theta_\rho$ dense in the representation space. Then, for any $A\in\text{M}\big(n,2g;\T\big)$ and for any $\varepsilon>0$ the mapping class group orbit intersects the open set $B_\varepsilon(A)\subset \textnormal{M}\big(n,2g;\T\big)$, \emph{i.e.} there is an element $g\in\textnormal{Mod}(S)$ such that $g^{-1}\cdot\Theta_\rho=\Theta_\rho\,g\in B_\varepsilon(A)$. In particular, $\big| \big| A-\Theta_\rho\, g\big| \big|<\varepsilon\textnormal{ mod } 2\pi$. Once again, by Proposition \ref{torga}, the matrix $g$ can be taken in $\spz$ and so Theorem \ref{thmbf} implies Theorem \ref{thmbf4} as desired. \qedhere \\
\end{proof}

\appendix

\section{Dense Orbits and further discussion}\label{mtoht}

\noindent In this appendix we are going to prove Theorem \ref{thmbf} for almost every representation without relying on Ratner's Theorem. We begin consider the genus one case and we shall use it to extend the discussion to surfaces of arbitrary genus.

\subsection{Direct proof of Theorem \ref{thmbf2} for almost every representations}\label{pmtgoc} The set of matrices $\textnormal{M}\big(n,2g;\T\big)$ contains, as a proper subset, the space $\mathcal{D}$ of all of those matrices of the following form
\begin{equation}\label{specmat}
\begin{pmatrix}
\theta_{1} &  \theta_{2} & \cdots  & \theta_{2i-1} & \theta_{2i} & \cdots & \theta_{2g-1} & \theta_{2g}\\
\vdots &  \vdots & & \vdots & \vdots & & \vdots &\vdots\\
\lambda_j\theta_{1} &  \lambda_j\theta_{2} & \cdots & \lambda_j\theta_{2i-1} & \lambda_j\theta_{2i} & \cdots & \lambda_j\theta_{2g-1} &  \lambda_j\theta_{2g}\\
\vdots & \vdots & & \vdots & \vdots & & \vdots &\vdots\\
\lambda_n\theta_{1} &\lambda_n\theta_{2} & \cdots  & \lambda_n\theta_{2i-1} & \lambda_n\theta_{2i} & \cdots & \lambda_n\theta_{2g-1} & \lambda_n\theta_{2g}\\
\end{pmatrix}
\end{equation}
\text{}\\
where $\big(\theta_1,\,\theta_2,\,\dots,\theta_{2g-1},\theta_{2g}\big)\in\R^{2g}\setminus\pi\Q^{2g}$ is the lift of $\big(e^{i\theta_1},\,e^{i\theta_2},\,\dots, e^{i\theta_{2g-1}},\,e^{i\theta_{2g}}\big)\in\T^{2g}$ contained in $[0,2\pi)^{2g}$ and the reals $\{1,\,\lambda_2,\,\dots,\lambda_n\}\subset\R$ are linearly independent over $\Q$. 

\begin{lem}
$\mathcal{D}$ is dense in $\textnormal{M}\big(n,2g;\T\big)$.
\end{lem}

\begin{proof}
Let $\lambda_2,\dots,\lambda_n$ real numbers such that $1,\lambda_2,\dots,\lambda_n$ are linearly independent over $\Q$. Let us consider the mapping $\varphi:\R^{2g}\longrightarrow \T^{2gn}$ defined as
\[ \Big(\theta_1,\theta_2,\dots,\theta_{2g-1},\theta_{2g}\Big)\mapsto \Big(\big(e^{i\theta_1},\dots,e^{i\theta_{2g}}\big),\big(e^{i\lambda_2\theta_1},\dots,e^{i\lambda_2\theta_{2g}}\big),\dots,\big(e^{i\lambda_n\theta_1},\dots,e^{i\lambda_n\theta_{2g}}\big)\Big).
\]  This mapping factors through a mapping $\overline{\varphi}:\R^{2g}\longrightarrow \R^{2gn}$ such that $\varphi=\exp\,\circ\,\overline{\varphi}$ and $\exp$ is the exponential mapping thought as in equation \eqref{projexp} introduced in section \S\ref{matrep}. The image of $\overline{\varphi}$ is a $2g$-dimensional linear subspace. Since the reals $\{1,\lambda_2,\dots,\lambda_n\}$ are linearly independent over $\Q$, then the projection via the exponential mapping is dense in $\T^{2gn}$. The space $\mathcal{D}$ is defined as the union of the images for each possible subset $\{\lambda_2,\dots,\lambda_n\}\subset\R$ such that $1,\lambda_2,\dots,\lambda_n$ are linearly independent over $\Q$. Therefore $\mathcal{D}$ is dense.
\end{proof}

\noindent The following lemma is easy to establish and the proof is left to the reader.

\begin{lem}
$\mathcal{D}$ is $\spz$-invariant.
\end{lem}

\noindent Let us consider first surfaces of genus one. Let $T$ be the genus one surface, let $\rho:\pi_1T\longrightarrow \T^n$ be a dense representation and let $\Theta_\rho$ be its associated matrix with respect to some basis $\big\{\alpha,\beta\big\}$ and the standard lattice of $\R^n$. Let $\Omega_\rho$ be the $\slz$-orbit of $\Theta_\rho$ in $\text{M}\big(n,2;\T\big)$. The associated matrix $\Theta_\rho$ has the following form:

\begin{equation}\label{c11}
\Theta_\rho=\begin{pmatrix}
\theta_{1} &  \theta_{2}\\
\vdots &  \vdots\\
\lambda_i\theta_{1} &  \lambda_i\theta_{2}\\
\vdots & \vdots \\
\lambda_n\theta_{1} &\lambda_n\theta_{2}
\end{pmatrix}
\end{equation}
where $\big(\theta_1,\,\theta_2\big)\in\R^2$ is the lift of $\big(e^{i\theta_1},\,e^{i\theta_2}\big)\in\T^2$ contained in $[0,2\pi)^2$ and $\lambda_i\in\R$, for any $i=2,\dots,n$, are linearly independent over $\Q$. Set
\[ \overline{\Theta}_\rho=\begin{pmatrix}
\Theta_\rho \\ 
\pi\cdot\text{I}_{2}.
\end{pmatrix}
\]

\noindent Since the representation $\rho$ is assumed to have a dense image, the matrix $\overline{\Theta}_\rho$ has maximal row rank, that is $\text{rk}_{\Z}=n+2$. This implies  the following properties of the matrix $\Theta_\rho$ above.\\

\SetLabelAlign{center}{\null\hfill\textbf{#1}\hfill\null}
\begin{enumerate}[label=\text{\textnormal{\ref{c11}}}\textnormal{.\roman*}, leftmargin=3em, labelwidth=2.5em, align=center, itemsep=\parskip]
\item The real numbers $\theta_1$ and $\theta_2$ cannot be both elements of $\pi\Q$. If this were the case, the row rank of the matrix $\overline{\Theta}_\rho$ would fail to be maximal, contradicting our assumptions. In the case one on them is an element of $\pi\Q$, we can change the basis in such a way they are both elements of $\R\setminus\pi\Q$. Indeed, assume without loss of generality that $\theta_2\in\pi\Q$. The Dehn twist $\mathcal{T}_\alpha$ along $\alpha$ maps the curve $\beta$ to $\alpha\beta$ and hence $\rho(\beta)$ is mapped to $\rho(\alpha\beta)$. The second column of $\Theta_\rho$ changes accordingly and the element of place $(1,2)$ of $\Theta_{\mathcal{T}_\alpha\cdot\rho}$ is nothing else that $\theta_1+\theta_2$. As $\theta_1\notin\pi\Q$ the same necessarily holds for $\theta_1+\theta_2$. In what follows, we shall assume both $\theta_1,\theta_2\notin\pi\Q$.\\
\item The real numbers $\pi,\theta_1,\dots,\lambda_i\theta_1,\dots,\lambda_n\theta_1$ are linearly independent over $\Q$. Indeed, if this were not the case then one can easily check that $\overline{\Theta}_\rho$ has not maximal rank. This implies that the subgroup of $\T^n$ generated by the vector $\big(\theta_1,\dots,\lambda_i\theta_1,\dots,\lambda_n\theta_1\big)$ is dense in $\T^n$, see \cite[Exercise 1.13]{BM}, meaning that $\T^n$ is monothetic (that is $\T^n$ contains a dense cyclic subgroup). The same holds also for the real numbers $\pi,\theta_2,\dots,\lambda_i\theta_2,\dots,\lambda_n\theta_2$.\\
\end{enumerate}

\noindent We begin with considering the $\slz$ action on the space $\text{M}\big(n,2;\R\big)$ seen as the universal cover of $\text{M}\big(n,2;\T\big)$, see also Remark \ref{mr}. Given the matrix $\Theta_\rho$ as in \ref{c11}, there is a unique lift, say $\Theta(\rho)$, in $\text{M}\big(n,2;\R\big)$ which is still of the form \ref{c11}. Notice that such a matrix is the unique one who has all the entries in the interval $[0,\,2\pi)$. Let us finally denote with $\Omega(\rho)$ the $\slz$-orbit of $\Theta(\rho)$ in $\text{M}\big(n,2;\R\big)$.\\

\noindent Since $\Theta(\rho)$ is of the form \ref{c11}, an easy computation shows that the matrix $A\cdot\Theta(\rho)\in\Omega(\rho)$ is still of the form \ref{c11} for any $A\in\slz$, that is the $i$-th row of $A\cdot\Theta(\rho)$ is $\lambda_i$-times $\big(\theta_1,\,\theta_2\big)A^{-1}$. Therefore, we can deduce that $\Omega(\rho)$ is contained in some proper linear subspace $S$ of $\R^{2n}$. In fact, the coefficients of any matrix $A\cdot\Theta(\rho)\in\Omega(\rho)$ satisfy the following homogeneous linear system
\begin{equation}  \mathcal{S}:
\begin{cases}
\begin{aligned}
\theta_{2,1}-&\lambda_2\theta_{1,1}&=0\\
&\vdots&\\
\theta_{n,1}-&\lambda_n\theta_{1,1}&=0\\
\theta_{22}-&\lambda_2\theta_{12}&=0\\
&\vdots&\\
\theta_{n2}-&\lambda_n\theta_{12}&=0
\end{aligned}
\end{cases}
\end{equation}
\noindent in $2n-2$ equations and $2n$ variables. Hence, $S$ is defined as the full space of solutions of the linear system $\mathcal{S}$. Let us consider then the subspace $S$. Since each of the $\lambda_i$ is taken as an element of $\R\setminus\Q$, the subspace $S$ necessarily meets the lattice $\text{M}\big(n,2;2\pi\Z\big)$ only at the origin. Therefore, the projection of the subspace $S$ into the space $\text{M}\big(n,2;\T\big)$ densely fills a closed subspace $K$ of $\text{M}\big(n,2;\T\big)$. We finally claim that $K$ cannot be a proper subspace. We notice that, due to the nature of the linear system $\mathcal{S}$, the subspace $S$ splits as the direct product $V_1\times V_2$ inside the vector space $\R^n\times\R^n\cong\text{M}\big(n,2;\R\big)$. Therefore, the image of $S$ into the space $\text{M}\big(n,2;\T\big)$ lies inside a closed subgroup of the form $H_1\times H_2$, where $H_i<\text{M}\big(n,1;\T\big)\cong\T^n$, for $i=1,2$. Notice that $K$ is a proper subgroup of $\text{M}\big(n,2;\T\big)$ if and only if $H_i$ is a proper subgroup of $\text{M}\big(n,1;\T\big)$. Therefore the proof of the final claim boils down to show that $H_i$ cannot be a proper subgroup for both $i=1,2$. As the group $H_1$ contains the vector $\exp\big(\theta_1,\lambda_2\theta_1,\dots,\lambda_n\theta_1\big)$, then it contains also the subgroup $\big\{\exp\big(t\big(\theta_1,\lambda_2\theta_1,\dots,\lambda_n\theta_1\big)\big)\,|\, t\in\Z\big\}$ and thus its closure which we know to be equal to the full space $\T^n$. In the same fashion, we can prove $H_2=\T^n$. Therefore $K=\text{M}\big(n,2;\T\big)$ and the $\slz$-orbit of $\Theta_\rho$ is dense in $\text{M}\big(n,2;\T\big)$ as desired.\\

\noindent The general case for surfaces of genus greater than one works similarly and we rely on the following

\begin{rmk}Let $S$ be any surface of genus $g\ge2$. The representation space \textnormal{Hom}$\big(\pi_1(S),\T^n\big)$ naturally splits as the direct sum of $g$ copies of \textnormal{Hom}$\big(\pi_1T,\T^n\big)$  where $T$ denote the $2$-torus (we do not use here the blackboard notation since $T$ is considered only as a topological surface regardless of its group structure). The basis $\big\{\alpha_1,\beta_1,\dots,\alpha_g,\beta_g\big\}$ of $\pi_1(S)$ we fixed satisfies moreover to the equalities $i\big(\alpha_i, \alpha_j\big)=i\big(\beta_i, \beta_j\big)=0,$ and $i\big(\alpha_i, \beta_j\big)=\delta_{ij}$ for all $i,j$ with $1\le i,j\le g$. We may associate to any representation $\rho:\pi_1(S)\longrightarrow \T^n$ the $g$-tuple of representations $\big(\rho_1,\dots,\rho_g\big)$ where $\rho_i$ is the restriction of $\rho$ to the handle generated by $\alpha_i,\beta_i$. Such a mapping defines then an isomorphism 
\begin{equation}
    \textnormal{Hom}\big(\pi_1(S),\T^n\big)\cong \bigoplus_{i=1}^g \textnormal{Hom}\big(\langle\alpha_i,\beta_i\rangle,\T^n\big).
\end{equation}
\noindent which depends on the basis chosen. This decomposition is a consequence of the fact that a surface of genus $g$ is the connected sum of a surface of genus $g-1$ and a torus $T$ along with the property that each representation $\rho$ sends all simple closed separating curves to the identity. A recursive argument leads to the desire conclusion. We finally notice that the group $\slz\times\cdots\times\slz<\spz$ introduced in Remark \ref{specsubgroup} acts diagonally on this space; the $i-$th copy of $\slz$ acts on \textnormal{Hom}$\big(\langle\alpha_i,\beta_i\rangle,\T^n\big)$.
\end{rmk}

\noindent Given any matrix of the form as in equation \eqref{specmat}, up to change the matrix with any element of $\spz$ we may assume, without loss of generality, that at least one of $\theta_{2i-1},\,\theta_{2i}\notin\pi\Q$. Under this condition all the observations A.2.(i-ii-iii) hold for each pair of columns 
\begin{equation}
\begin{pmatrix}
\theta_{2i-1} &  \theta_{2i}\\
\vdots &  \vdots\\
\lambda_j\theta_{2i-1} &  \lambda_j\theta_{2i}\\
\vdots & \vdots \\
\lambda_n\theta_{2i-1} &\lambda_n\theta_{2i}
\end{pmatrix}
\end{equation}

\noindent Therefore the action of the $g$-times product $\slz\times\cdots\times\slz<\spz$ provides a dense orbit inside the space $\text{M}\big(n,2g;\T\big)$ as desired. 

\subsection{Finding curve generating dense subgroups}\label{fdv} All the representations $\rho$ considered in subsection \ref{pmtgoc} above are characterized by the following property: Each column of the associated matrix $\Theta_\rho$ generates a dense subgroup of $\T^n$. It turns out that for any such a representation one can find infinitely many curves whose image generates a dense subgroup in $\T^n$. This lead the authors to ask themselves: Given a dense representation $\rho:\pi_1(S)\longrightarrow \T^n$, can we find a simple closed curve $\gamma$ such that $\langle\rho(\gamma)\rangle$ is dense in $\T^n$? Remind that a vector $\big(e^{i\theta_1},\,\dots,\,e^{i\theta_{n}}\big)\in\T^{n}$ generates a dense subgroup if and only if $\pi,\theta_1,\dots,\theta_n$ are linearly independent over $\Q$. As a corollary of Lemma \ref{Lemma 4.1} we deduce the following.

\begin{lem}
Let $\rho:\pi_1(S)\longrightarrow \Bbb S^1$ be a dense representation. Then there always exists a simple closed curve $\gamma$ such that $\overline{\langle\rho(\gamma)\rangle}=\Bbb S^1$.
\end{lem}

\noindent However, for $n\ge2$, the scenario changes completely. Indeed, for any $n$ we can find examples of dense representations which do not have any curve generating a dense subgroup in $\T^n$.

\begin{ex} Let $S$ be a surface of genus $g$ and $\rho:\pi_1(S)\longrightarrow \Bbb T^2$ be the representation having as the associated matrix $\Theta_\rho\in \text{M}\big(2,2g;\T\big)$ the following
\[
\begin{pmatrix}
1 & 1 & 0 & 0 & 0 & 0 & \cdots & \cdots & 0 & 0\\
0 & 0 & 1 & 1 & 0 & 0 & \cdots & \cdots & 0 & 0
\end{pmatrix}\in \text{M}\big(2,2g;\T\big).
\] One can show that $\overline{\Theta}_\rho$ has maximal rank and hence $\rho$ is a dense representation. However, no curve is applied by $\rho$ to a vector generating a dense subgroup.
\end{ex}

\begin{ex} Let $S$ be a surface of genus $g$ and $\rho:\pi_1(S)\longrightarrow \Bbb T^n$ be the representation having as the associated matrix $\Theta_\rho\in \text{M}\big(2,2g;\T\big)$ the following
\[
\begin{pmatrix}
1 & 1 & 0 & 0 & 0 & 0 & \cdots & \cdots & 0 & 0\\
0 & 0 & 1 & 1 & 0 & 0 & \cdots & \cdots & 0 & 0\\
\theta_3 & 0 & 0 & 0 & 0 & 0 & \cdots & \cdots & 0 & 0\\
\vdots &  &  &  &  &  &  &  &  & \vdots\\
\theta_n & 0 & 0 & 0 & 0 & 0 & \cdots & \cdots & 0 & 0
\end{pmatrix}\in \text{M}\big(n,2g;\T\big),
\] where $1,\theta_3,\dots,\theta_n$ are linearly independent over $\Q$. One can show that $\overline{\Theta}_\rho$ has maximal rank and hence $\rho$ is a dense representation. However, no curve is applied by $\rho$ to a vector generating a dense subgroup.\\
\end{ex}

\section{Surfaces with one puncture}\label{swpc}
\noindent Let us now discuss the case of the one-holed torus $\Sigma$. We shall denote $\pi_1\Sigma\cong\langle \alpha,\beta\rangle$ the fundamental group of $\Sigma$. Also in this case the choice of a representation consists in choosing for each generator an element of $\T^n$. The representation space \textnormal{Hom}$\big(\pi_1\Sigma,\T^n\big)$ trivially identifies with the space $\T^n\times \T^n$. For each choice of an element $\bf{c}$, the relative representation variety \textnormal{Hom}$_{\bf{c}}\big(\pi_1\Sigma,\T^n\big)$ is defined as the preimage of $\bf{c}$ via the commutator map $k:\T^n\times \T^n\to\T^n$. Thus, as a consequence of the abelian property, the relative representation space is empty for any $\textbf{c}\neq (1,\dots,1)$ and coincides with the full representation variety when $\textbf{c}= (1,\dots,1)$. Once again, the action of $\T^n$ by inner automorphisms is trivial and hence the character variety trivially coincides with the representation space. As a consequence, the space \textnormal{Hom}$\big(\pi_1\Sigma,\T^n\big)$ naturally identifies with the space \textnormal{Hom}$\big(\pi_1T,\T^n\big)$. The equalities $\text{Mod}(T)=\text{Mod}(\Sigma)=\text{SL}(2,\Z)$ are well-known and the actions of Mod$(T)$ and Mod$(\Sigma)$ on the representation spaces associated to $T$ and $\Sigma$ respectively coincide. Therefore, we have the following proposition.

\begin{prop}\label{equivprop}
Theorem \ref{thmbf} and Theorem \ref{thmbf2} hold for the torus $T$ if and only if they hold for the one-holed torus $\Sigma$.
\end{prop}

\noindent More generally, the main results of the present work extend to surfaces of higher genus and with one boundary component. Indeed, let $S_{g,1}$ be a surface a surface of genus $g$ and one boundary component. We have already seen above that this is true for the one-holed torus $\Sigma$, Proposition \ref{equivprop}. The general claim follows because, since$\T^n$ is abelian, one can establish an identification between the representations spaces \textnormal{Hom}$\big(\pi_1(S),\T^n\big)$ and \textnormal{Hom}$\big(\pi_1\big(S_{g,1}\big),\T^n\big)$. Since the mapping class group coincides with the pure mapping class group for one-puncture surfaces the following proposition also holds.

\begin{prop}\label{equivprop2}
Theorem \ref{thmbf} and Theorem \ref{thmbf2} hold for a closed surface of genus $g$ if and only if they hold for the one-holed surface of genus $g$.\\
\end{prop}

\bibliographystyle{amsalpha}
\bibliography{biblio}

\providecommand{\bysame}{\leavevmode\hbox to3em{\hrulefill}\thinspace}
\providecommand{\MR}{\relax\ifhmode\unskip\space\fi MR }
\providecommand{\MRhref}[2]{%
  \href{http://www.ams.org/mathscinet-getitem?mr=#1}{#2}
}
\providecommand{\href}[2]{#2}
\begin{thebibliography}{BGMW19}

\bibitem[BGMW19]{BGMW}
I.~Biswas, S.~Gupta, M.~Mj, and J.P. Whang, \emph{Surface group representations
  in $\slc$ with finite mapping class orbits}, \\
  https://arxiv.org/abs/1707.00071v3 (2019).

\bibitem[BKMS18]{BKMS}
Indranil Biswas, Thomas Koberda, Mahan Mj, and Ramanujan Santharoubane,
  \emph{Representations of surface groups with finite mapping class group
  orbits}, New York J. Math. \textbf{24} (2018), 241--250. \MR{3778503}

\bibitem[BM00]{BM}
M.~Bachir Bekka and Matthias Mayer, \emph{Ergodic theory and topological
  dynamics of group actions on homogeneous spaces}, London Mathematical Society
  Lecture Note Series, vol. 269, Cambridge University Press, Cambridge, 2000.
  \MR{1781937}

\bibitem[Bou20]{YB}
Yohann Bouilly, \emph{On the torelli group action on compact character
  varities.}, arxiv.org/abs/2001.08397 (2020).

\bibitem[BtD95]{BtD}
Theodor Br\"{o}cker and Tammo tom Dieck, \emph{Representations of compact {L}ie
  groups}, Graduate Texts in Mathematics, vol.~98, Springer-Verlag, New York,
  1995, Translated from the German manuscript, Corrected reprint of the 1985
  translation. \MR{1410059}

\bibitem[Cas97]{CJ}
J.W.S. Cassels, \emph{An introduction to the geometry of numbers},
  Springer-Verlag, Berlin, 1997.

\bibitem[GLX19]{GLX}
W.~M. Goldman, S.~Lawton, and E.~Z. Xia, \emph{Mapping class action on
  su$(3)$-character varieties}, https://arxiv.org/abs/1909.10968 (2019).

\bibitem[Gol84]{GO84}
William~M. Goldman, \emph{The symplectic nature of fundamental groups of
  surfaces}, Adv. in Math. \textbf{54} (1984), no.~2, 200--225. \MR{762512}

\bibitem[Gol88]{GO88}
W.~M. Goldman, \emph{Topological components of spaces of representations},
  Invent. Math. \textbf{93} (1988), no.~3, 557--607. \MR{952283}

\bibitem[Gol97]{GO97}
William~M. Goldman, \emph{Ergodic theory on moduli spaces}, Ann. of Math. (2)
  \textbf{146} (1997), no.~3, 475--507. \MR{1491446}

\bibitem[Gol06]{GO06}
\bysame, \emph{Mapping class group dynamics on surface group representations},
  Problems on mapping class groups and related topics, Proc. Sympos. Pure
  Math., vol.~74, Amer. Math. Soc., Providence, RI, 2006, pp.~189--214.
  \MR{2264541}

\bibitem[HR63]{HR}
Edwin Hewitt and Kenneth~A. Ross, \emph{Abstract harmonic analysis. {V}ol. {I}:
  {S}tructure of topological groups. {I}ntegration theory, group
  representations}, Die Grundlehren der mathematischen Wissenschaften, Bd. 115,
  Academic Press, Inc., Publishers, New York; Springer-Verlag,
  Berlin-G\"{o}ttingen-Heidelberg, 1963. \MR{0156915}

\bibitem[Joh80]{JD}
Dennis Johnson, \emph{Conjugacy relations in subgroups of the mapping class
  group and a group-theoretic description of the rochlin invariant}, Math. Ann.
  \textbf{249} (1980), no.~3, 243--263.

\bibitem[MW16]{MW}
Julien March\'e and Maxime Wolff, \emph{The modular action on
  {$\pslr$}-characters in genus 2}, Duke Mathematical Journal \textbf{165}
  (2016), no.~2, 371--412.

\bibitem[MW19]{MW2}
Julien March\'{e} and Maxime Wolff, \emph{Six-point configurations in the
  hyperbolic plane and ergodicity of the mapping class group}, Groups Geom.
  Dyn. \textbf{13} (2019), no.~2, 731--766. \MR{3950649}

\bibitem[Nie27]{NI27}
Jakob Nielsen, \emph{Untersuchungen zur {T}opologie der geschlossenen
  zweiseitigen {F}l\"{a}chen}, Acta Math. \textbf{50} (1927), no.~1, 189--358.
  \MR{1555256}

\bibitem[PX00]{PrX1}
Joseph~P. Previte and Eugene~Z. Xia, \emph{Topological dynamics on moduli
  spaces. {I}}, Pacific J. Math. \textbf{193} (2000), no.~2, 397--417.
  \MR{1755824}

\bibitem[PX02]{PrX2}
\bysame, \emph{Topological dynamics on moduli spaces. {II}}, Trans. Amer. Math.
  Soc. \textbf{354} (2002), no.~6, 2475--2494. \MR{1885660}

\bibitem[PX03a]{PX2}
Doug Pickrell and Eugene~Z. Xia, \emph{Ergodicity of mapping class group
  actions on representation varieties. ii. surfaces with boundary.}, Transform.
  Groups \textbf{8} (2003), no.~4, 397--402.

\bibitem[PX03b]{PX3}
J.P. Previte and E.~Z. Xia, \emph{Exceptional discrete mapping class group
  orbits in moduli spaces}, Forum Math. \textbf{15} (2003), no.~6, 949--954.

\bibitem[Rat91]{RM}
M.~Ratner, \emph{Raghunathan's topological conjecture and distributions of
  unipotents flows.}, Duke Mathematical Journal \textbf{63} (1991), no.~1,
  235--280.

\end{thebibliography}
\end{document}